\documentclass[11pt]{amsart}
\usepackage{fullpage,amsmath, amssymb,amsfonts,amsmath,latexsym,amscd,amsthm, mathtools, stackrel}
\usepackage{tikz-cd}
\usepackage[top=2cm, bottom=4.5cm, left=3cm, right=3cm]{geometry}
\usepackage{float, graphicx}
\graphicspath{ {./images/} }

\usepackage{hyperref}

\usepackage{tikz-cd}
\usepackage{tikz}
\usepackage{graphicx,subfigure}

\usepackage{enumerate}
\usepackage[shortlabels]{enumitem}

\newtheorem{theorem}{Theorem}
\numberwithin{theorem}{section}
\newtheorem{proposition}[theorem]{Proposition}
\newtheorem{lemma}[theorem]{Lemma}
\newtheorem{claim}[theorem]{Claim}

\newtheorem{corollary}[theorem]{Corollary}
\newtheorem{question}[theorem]{Question}

\theoremstyle{definition}
\newtheorem{example}[theorem]{Example}
\newtheorem{remark}[theorem]{Remark}
\newtheorem{definition}[theorem]{Definition}

\makeatletter

\newcommand{\g}{\mathfrak{g}}
\newcommand{\ord}{\text{ord}}
\newcommand{\Id}{\text{Id}}
\newcommand{\iterates}{\{1,2,3,4,6, 8, 12\}}
\newcommand{\V}{\mathcal{V}}

\newcommand{\OO}{\mathcal{O}}

\newcommand{\F}{\mathbb F}
\newcommand{\A}{\mathbb A}
\newcommand{\Z}{\mathbb Z}

\newcommand{\C}{\mathbb C}

\newcommand{\bP}{\mathbb P}
\newcommand{\T}{\mathcal{T}}
\newcommand{\Q}{{\mathbb Q}}
\newcommand{\R}{{\mathbb R}}

\newcommand{\bG}{\mathbb G}

\newcommand{\indet}{\text{Indet}}

\newcommand{\inj}{\hookrightarrow}
\newcommand{\lra}{\longrightarrow}
\newcommand{\dra}{\dashrightarrow}

\makeatletter

\newcommand{\End}{\text{End}}

\title{Rational self-maps of projective surfaces with a regular iterate}
\author{Sina Saleh}
\begin{document}
\begin{abstract}
We show that if $\Phi: X \dra X$ is a dominant rational self-map of a projective surface $X$ over $\C$ with a regular and non-invertible iterate $\Phi^n$, then we can take $n \leq 12$.  This bound is sharp and realized on $X = \bP^2$. In the case where $\Phi$ is a birational self-map of $\bP^2$ we prove that as long as $\Phi$ does not preserve a non-constant fibration, if some iterate $\Phi^n$ is regular then $\Phi$ itself must be regular.
\end{abstract}
\maketitle

\section{Introduction}

In \cite{BGR}, Bell, Ghioca, and Reichstein investigated dominant rational self-maps \(\Phi\) of a variety \(X\) that admit a regular iterate. We refer to such maps as \textit{eventually regular}. The authors suggest in \cite{BGR} that imposing a suitable rigidity hypothesis on the variety \(X\) leads to strong structural conclusions about an endomorphism \(\Phi\) when it is eventually regular. In particular, \cite[Theorem 1.1(a)]{BGR} shows that if \(\Phi\) is an eventually regular dominant rational self-map of a semiabelian variety \(G\) over an algebraically closed field of characteristic \(0\), then either \(\Phi\) preserves a non-constant fibration or \(\Phi\) itself must be regular.  

This paper on the other hand focuses on rational self-maps of smooth projective surfaces over \(\mathbb{C}\) that are eventually regular without imposing additional assumptions on the type of surface. A classic example is the standard Cremona transformation $\sigma([x:y:z]) = [yz:xz:xy]$ on $\bP^2$ which is not regular but whose second iterate is the identity. We will investigate the structure of eventually regular self-maps with a regular but non-invertible iterate. We provide explicit examples in Section 2.  
\begin{theorem}
\label{thm:main}
Let $\Phi: \bP^2 \dra \bP^2$ be a dominant rational self-map over $\C$. Assume that $\Phi^k$ is a non-invertible endomorphism for some $k \ge 1$.  Then, $\Phi^i$ is an endomorphism for some $i \in \iterates$. Also, for every $i \in \iterates$ there exists a monomial rational self-map $\Phi: \bP^2 \dra \bP^2$ over $\C$ whose first iterate that becomes a non-invertible endomorphism of $\bP^2$ is $\Phi^i$. 
\end{theorem}

In the case where $\Phi$ is birational we prove the next theorem.
\begin{theorem}
\label{thm:birational}
Let $\Phi: \bP^2 \dra \bP^2$ be a birational self-map over $\C$. Assume that  $\Phi^k$ is an endomorphism for some $k \ge 1$. Moreover, assume that there does not exist a non-constant rational fibration $f: \bP^2 \dra \bP^1$ such that $f \circ \Phi = f$. Then, $\Phi$ itself is regular. 
\end{theorem}
\begin{remark}
\label{rem:no-fibration-pres-nec}
Example \ref{ex:bir-large-order} illustrates that if we allow $\Phi$ to preserve a non-constant fibration, then $\Phi$ could take any arbitrary number of iterations to become regular. So, there is no hope of giving a bound similar to Theorem \ref{thm:main}. This explains why the extra assumption that no non-constant rational fibration is preserved is necessary in Theorem \ref{thm:birational}.   
\end{remark}

The next theorem is a generalization of the first part of Theorem \ref{thm:main} to the case of all complex projective surfaces.
\begin{theorem}
\label{thm:more-precise}
Let $X$ be a smooth projective surface over $\C$ and let $\kappa(X)$ denote the Kodaira dimension of $X$. Suppose that $\Phi$ is a dominant rational self-map of $X$ with a regular iterate that is not invertible. Then, one of the following must hold:
\begin{enumerate}
    \item[(i)] $\kappa(X) \ge 0$, $X$ is minimal and $\Phi$ itself is regular, or
    \item[(ii)] $X$ is a $\bP^1$-bundle over a curve $C$ of genus at least 1 and $\Phi$ itself is regular, or

    \item[(iii)] $X$ is a toric surface and $\Phi^i$ is regular for some $i \in \iterates$. 
\end{enumerate}
\end{theorem}
\begin{remark}
\label{rem:vac-when-abelian}
Theorem \ref{thm:more-precise} is vacuous when $X$ is an abelian or hyperelliptic surface. To see this, note that abelian and hyperelliptic surfaces do not contain rational curves (see \cite[Theorem 4.1 and Corollary 4.2]{surf-survey}) and any morphism to a surface $X$ that contains no rational curves must be regular. Indeed, if it were not, then by resolution of singularities we could resolve its indeterminacies by a sequence of blow-ups, and the images of the resulting exceptional divisors would be rational curves in $X$, contradicting our assumption.  So, $\Phi$ must already be regular for abelian and hyperelliptic surfaces. 
\end{remark}
\begin{remark}
Remark \ref{rem:vac-when-abelian} along with \cite[Theorem 3.2]{Fujimoto} show that the only interesting case of Theorem \ref{thm:more-precise} when $\kappa(X) \ge 0$ is when $X$ is an elliptic surface with Kodaira dimension equal to 1. There are examples of such surfaces that admit rational self-maps that are not regular everywhere. For example, let $X$ be an elliptic surface such that $\pi: X 
\lra C$ has singular fibers and let $E$ be the generic fiber of $\pi$. Then, any endomorphism of $E$, which is an elliptic curve over $\C(C)$, gives rise to a rational self-map of $X$ which has indeterminacy at the singular fibers of $\pi$. So, when $\kappa(X) \ge 0$, Theorem 1.4 will be non-trivial in the case of elliptic surfaces of $\kappa(X) = 1$ with singular fibers.
\end{remark}

\begin{remark}
In contrast to Theorem \ref{thm:main} which only treats the characteristic zero case, \cite[Theorem 1.1]{BGR} classifies eventually regular dominant rational self-maps of semiabelian varieties defined over a field of positive characteristic as well.
\end{remark}

\begin{remark}
\label{rem:birational-exc}
Example \ref{ex:non-minimal} provides a counterexample to the possible generalization of Theorem \ref{thm:birational} when $\Phi$ is allowed to be a birational self-map of an arbitrary complex projective surface. Having said that, there are surfaces other than $\bP^2$ where Theorem \ref{thm:birational} still holds.  For example, if $X$ is a minimal complex projective surface with $\kappa(X) \ge 0$, then by \cite[Proposition 7.5]{diller-favre}, any dominant rational self-map of $X$ is already regular. So, Theorem \ref{thm:main} is vacuously true. Also, even though we have not treated the case of Hirzebruch surfaces for Theorem \ref{thm:birational}, it is worth mentioning that we still expect Theorem \ref{thm:birational} to hold for these surfaces. 

\end{remark}
\begin{remark}
\label{rem:good-model-bir}
As explained in Remark \ref{rem:birational-exc}, Theorem \ref{thm:birational} does not hold for arbitrary complex projective surfaces. Having said that, \cite[Proposition 2.12]{favre} shows that if $\Phi: X \dra X$ is birational and eventually regular, then there is a smooth projective surface $\hat{X}$, a birational morphism $\pi: \hat{X} \lra X$, and an automorphism $F: \hat{X} \lra \hat{X}$ such that  $\Phi \circ \pi = \pi \circ F$. This seems to be the most general conclusion one can reach in the birational setting regardless of the type of the surface $X$. 

In the case of a general eventually regular self-map, we can always find a normal model $\hat{X}$ such that $\Phi$ lifts to an endomorphism of $\hat{X}$ (see Section \ref{sec:good-normal-model}). 
\end{remark}

Theorem \ref{thm:more-precise} motivates the following question.

\begin{question}
\label{question:boundIterate}
Let $X$ be a variety over $\C$.  Let $\Phi: X \dra X$ be a dominant rational self-map over $\C$ that is eventually regular. Moreover, assume that there does not exist a non-constant rational map $f: X \dra \bP^1$ such that $f \circ \Phi = f$. Does there exist a constant $\ell(X)$ such that $\Phi^i$ is already regular for some $i \in \{1,2,\dots,\ell(X)\}$? 
\end{question} 

As illustrated by Example \ref{ex:prod}, the hypothesis ruling out invariant fibrations is necessary. We thank Serge Cantat for pointing out this example.

\begin{remark}
 The answer to Question \ref{question:boundIterate} is also positive when \(X = \mathbb{A}^1\). It is well known that if a rational function of degree at least 2 over a field of characteristic \(0\) has a regular iterate—i.e., a polynomial iterate—then its second iterate must already be regular (see \cite[Section 4.1]{Bea91}). In arbitrary characteristic, \cite{Sil96} provides a complete classification of rational functions with polynomial iterates. 
\end{remark}
\begin{remark}
We note that Theorem \ref{thm:more-precise} also shows that there is a uniform bound (not depending on $X$) for the number of iterations of an eventually regular self-map to become regular under the assumption that $X$ is a smooth projective surface. So, one could naturally rephrase Question \ref{question:boundIterate} asking for a uniform bound depending on the dimension of $X$ only.   
\end{remark}

It turns out that if a dominant rational self-map $\Phi$ of a projective variety $X$ is not regular itself but has an iterate $\Phi^k$ that is regular and non-invertible, then there must exist a subvariety $W$ of $X$ that is totally invariant under $\Phi^k$ that is $\Phi^{-k}(W) = W$. More precisely, in the case of projective surfaces, we have the next proposition which is the analog of \cite[Proposition 1.2]{BGR} and will be instrumental in our proofs. 

\begin{proposition}
\label{prop:totally-inv}
Let $X$ be a smooth projective variety over $\C$ and $\Phi: X \dra X$ be a dominant self-map that is not regular. Suppose that $\Phi^k$ is regular for some $k \ge 2$. Then, there exists a proper non-empty subvariety $W$ of codimension at least 2 such that $\Phi^{-k}(W) = W$. Furthermore, if one of the following holds:
\begin{enumerate}
    \item $X = \bP^n$, or
    \item $X$ is a $\bP^1$-bundle over a smooth projective curve $C$,
\end{enumerate}
then there exists a subvariety $V$ of pure codimension 1 such that $\Phi^{-k}(V) = V$ and $\Phi$ induces a quasi-finite and surjective endomorphism on $X \setminus V$.
\end{proposition}
It is likely that Proposition \ref{prop:totally-inv} holds in a more general setting. However, the above version will be more than sufficient for our proof of Theorem \ref{thm:more-precise}. 

The main ingredient in the proof of \cite[Proposition 1.2]{BGR} is \cite[Lemma 4.1]{BGR}. Using \cite[Lemma 4.1]{BGR} one can show that the indeterminacy locus of $\Phi$ is totally invariant under $\Phi^k$. This subvariety must have codimension 1 by \cite[Proposition 1.3]{AR86}. Our proof of Proposition \ref{prop:totally-inv} uses a similar approach. We prove an analog of \cite[Lemma4.1]{BGR} (See Lemma \ref{lem:regularComposition}) in the case of rational self-maps between projective surfaces and use it to prove that the indeterminacy loci of all $\Phi^i$ for $i = 1,\dots,k-1$ are totally invariant under $\Phi^k$. But there are two new sources of difficulty: 
\begin{itemize}
    \item[($\ast$)] the indeterminacy locus of a non-regular rational self-map of a projective surface $X$ has codimension 2 and due to this the codimension of the subvariety $W$ will be 2, and

    \item [$(\ast\ast)$] The regular iterate $\Phi^k$ is no longer \'etale necessarily, which has to be the case when $X$ is a semiabelian variety.
\end{itemize}  As we will explain in the proof strategy for Theorem \ref{thm:more-precise}, the existence of a totally invariant subvariety of codimension 2 is sufficient for us to prove Theorem \ref{thm:more-precise} in the case of $\kappa(X) \ge 0$. In the case of ruled surfaces, especially when $X = \bP^2$, the existence of a subvariety $V$ of pure codimension 1 is crucial for our arguments. So, $(\ast)$ becomes a true issue when $X$ is a ruled surface. To construct $V$ in this case, we will also consider the preimages of the subvariety $W$ under the iterations of $\Phi$ (see equation \eqref{eqn:def-V} in Section \ref{sec:key-prop}). To deal with $(\ast \ast)$, we use a singular inverse function theorem \cite[Corollary 1.1.8]{JW15} which allows to us conduct a local analysis near the ramified points of the endomorphism $\Phi^k$.

\subsection{Proof strategy for Theorems \ref{thm:main} and \ref{thm:more-precise}} Note that the first conclusion of Theorem \ref{thm:main} is clearly a corollary of Theorem \ref{thm:more-precise}. Also, the second conclusion is a consequence of Examples \ref{ex:2-regular} through \ref{ex:3,4,6,8} given in Section \ref{sec:examples}. So, it suffices to show Theorem \ref{thm:more-precise} only. We assume that $\Phi^k$ is regular for some $k \ge 1$. Moreover, we can assume that $\Phi$ itself is not regular because we would be done otherwise. Proposition \ref{prop:totally-inv} then shows that there exists a totally invariant subvariety $W$ of pure codimension 2. Using the classification of projective surfaces admitting non-invertible endomorphisms (See Theorems \ref{thm:pos-kod} and \ref{thm:ruled}) we conclude that there are three main cases to consider:
\begin{enumerate}
    \item[\textbf{1.}] $\kappa(X) \ge 0$, $X$ is minimal and $\Phi^k$ is a finite \'etale endomorphism, or

    \item[\textbf{2.}] $X$ is a $\bP^1$-bundle over a curve $C$ of genus at least 1, or

    \item[\textbf{3.}] $X$ is a toric surface. 
\end{enumerate}

In case 1, since $W$ is a subset of finitely many points that are totally invariant under $\Phi^k$ and $\Phi^k$ is \'etale, we conclude that $\Phi^k$ is an automorphism which contradicts the hypothesis that $\Phi^k$ is non-invertible. In case 2, we use a theorem of Amerik, \cite[Theorem 1]{AmerikProjBundles} to deduce that $X$ must be a product of the form $C \times \bP^1$ after a finite base change. If $X$ is a product, this allows us to conclude that $\Phi^k$ is a product map (See Proposition \ref{prop:trivial-then-split}). As a consequence, since $W$ is totally invariant, the union of the fibers of the two projections $C\times \bP^1 \lra \bP^1$ and $C\times \bP^1 \lra C$ containing $W$ must also be totally invariant. Then, we can deduce that $\Phi$ is a skew-product inducing a polynomial endomorphism on all fibers of the projection $C\times \bP^1 \lra C$ (See Lemmas \ref{lem:rat-pres-fibers} and \ref{lem:add-more-totally-invs}). Using these nice forms of $\Phi$ and $\Phi^k$ we can then argue directly that the map $\Phi$ itself must be regular (See Lemma \ref{lem:skew-then-split}). When $X$ is not a product, the approach is similar but requires additional
subtle arguments.

In Case 3, we can use Proposition \ref{prop:totally-inv} to deduce that there exists a totally invariant subvariety $V$ pure codimension 1 such that $\Phi$ induces a quasi-finite and surjective endomorphism on $X \setminus V$. The most difficult situation to analyze is when the self-map $\Phi$ induces an endomorphism on $X \setminus V \simeq \C^2$ for some curve $V$ in $X$. To understand the smallest number of iterations it takes for $\Phi$ to become regular on $X$ we need to understand the period of the valuations corresponding to the prime divisors contained in $V$. The results in \cite{FJ07} about fixed valuations allow us to bound these periods. In recent years the valuative tree has proven to be a very useful tool in algebraic and arithmetic dynamics. To name a few articles that use it, we refer the reader to \cite{Jonsson-Wulcan-height, Xie-DML-aut, Xie-DML-End, Xie-ZDOC-C2, cantat-xie, Dang-Tiozzo-limit}. It is also worth mentioning that the valuative theory of the complex affine plane has recently been extended to all affine surfaces in \cite{abboud}. 

Let us explain in more detail the proof plan in Case 3. There are three main subcases to consider: 
\begin{itemize}
    \item[\textbf{3(a).}] $X = \bP^2$ or
    \item[\textbf{3(b).}] $X$ is a Hirzebruch surface $\F_n$ for some $n \ge 0$ or
    \item[\textbf{3(c).}] $X$ is obtained from $\F_n$ for some $n\ge 1$ after finitely many point blow-ups.  
\end{itemize}
The strategy for handling all three of these subcases are similar. To explain the strategy, we focus on case 3(a) and refer the reader to subsections \ref{subsec:Hirzebruch} and \ref{subsec:general-toric} for a detailed treatment of the remaining cases. Assuming $X = \bP^2$, we know by Proposition \ref{prop:totally-inv} that there is a subvariety $V$ of pure dimension 1 such that $\Phi^{-k}(V) = V$. By the linearity theorem for $\bP^2$ (See Theorem \ref{conj:lin-conj}), $V$ is a union of at most three hyperplanes. Using Proposition \ref{prop:totally-inv} we assume that $\Phi$ induces a quasi-finite and surjective endomorphism of an open subset $U := X \setminus V$ of $X$ that is isomorphic to $\A^2, \A^2 \setminus \{x = 0\}$ or $\A^2 \setminus \{xy = 0\} \simeq \bG_m^2$. When $U$ is isomorphic to $\A^2$ or $\A^2 \setminus \{x = 0\}$, we can assume that $U \simeq \A^2$ without loss of generality at the expense of replacing $\Phi$ with $\Phi^2$. A proposition by Favre and Jonsson about the fixed points of the map induced by $\Phi$ on the valuative tree at infinity (See Proposition \ref{prop:FJ-5.3}) allows us to conclude that $\Phi^2$ must already extend to an endomorphism of $X$.

The last case to consider is $U \simeq \bG_m^2$. This is the only case where we can have dominant rational self-maps taking $3,4,6,8,$ or $12$ iterations to become regular. For concrete examples see Examples \ref{ex:12-reg} and $\ref{ex:3,4,6,8}$. Since $\Phi$ induces an endomorphism of $U \simeq \bG_m^2$, it is given by a group endomorphism $\varphi$ composed with a translation (See \cite[Theorem 2]{Iitaka}). The group endomorphism $\varphi$ corresponds to an invertible matrix $A \in M_{2,2}(\Z)$. Since by our assumption, the lines $\{x = 0\}$ and $\{y = 0\}$ are totally invariant under an iterate of $\Phi$, we conclude that some power of $A$ must be a diagonal matrix with non-negative entries. We can then bound the smallest power of $A$ with this property and conclude that $A^i$ is already a diagonal matrix with non-negative entries for some $i \in \iterates$ (See Lemma \ref{lem:smallestPower}). It follows that $\Phi^i$ must already extend to $X$.

\subsection{Proof strategy for Theorem \ref{thm:birational}} Similar to the proof of Theorem \ref{thm:more-precise} we first assume that $\Phi$ itself is not regular and $\Phi^k$ is an automorphism of $\bP^2$ for some $k \ge 2$. We also assume that $\Phi$ does not preserve any non-constant fibrations. Then, using Proposition \ref{prop:totally-inv} we can find a subvariety $V$ of pure codimension 1 such that $\Phi^{-k}(V) = V$ and $\Phi$ induces an automorphism of $\bP^2 \setminus V$. Since $\Phi^k$ is an automorphism of $\bP^2$ and it does not preserve a non-constant fibration it is then not difficult to show that after a change of coordinates we may assume that $V = V(X_1X_2) \text{ or } V(X_0X_1X_2)$ (See Proposition \ref{prop:inv-subvarieties}). Using the fact that $\Phi$ is an automorphism of $\bP^2 \setminus V$ we conclude that $\Phi$ has a computationally convenient form (See Equations \eqref{eqn:form-of-phi-1} and \eqref{eqn:form-of-phi-2}). A straightforward computation using the form of $\Phi$ and using the fact that $\Phi^k$ is given by some matrix in $\text{PGL}_3(\C)$ we can then show that $\Phi$ itself must be an automorphism of $\bP^2$. 

\subsection{Notation and terminology.} Throughout this paper $\bP^n$ denotes the $n$-dimensional projective space over $\C$. We use $X$, $Y$, and $Z$ as the homogeneous coordinates of $\bP^2$ and $X_0,X_1,\dots,X_n$ as the homogeneous coordinates of $\bP^n$ when $n$ is arbitrary. Given a homogeneous polynomial $f \in \C[X_0,\dots,X_n]$, we let $Z(f) \subset \bP^n$ denote the vanishing locus of $f$. We identify $\C^2$ with the open subset $\bP^2 \setminus \{Z = 0\}$ of $\bP^2$ and use $x$ and $y$ to denote the standard coordinates of $\C^2$. We let $\A^n$ and $\bG_m^n$ denote the affine $n$-space and the $n$-dimensional algebraic torus, respectively. We use $x_1,\dots,x_n$ to denote the coordinates of both $\A^n$. Unless otherwise specified, we always view $\bG_m^n$ as an open subvariety of $\A^n$ defined by $\A^n \setminus Z(x_1\cdots x_n)$ and view $\A^n$ as an open subvariety of $\bP^n$ defined by $\bP^n \setminus Z(X_n)$. Any endomorphism of $\bG_m^n$ is given by a translation composed with a group endomorphism. If  $A := (a_{ij})_{1 \le i,j \le n} \in M_{n,n}(\Z)$ and $\vec{\lambda} = (\lambda_1,\dots,\lambda_n) \in \C^n$ are given, the associated endomorphism of $\bG_m^n$ is defined as 
\[
(x_1,\dots,x_n) \mapsto (\lambda_1x_1^{a_{11}}\cdots x_n^{a_{1n}}, \dots, \lambda_n x_1^{a_{n1}}\cdots x_n^{a_{nn}}). 
\]
For conciseness, we will also write this as 
\[
\vec{x} \mapsto \vec{\lambda} + A \cdot \vec{x}.
\]
In this context, we write $\vec{0}$ for the identity element $(1,\dots,1) \in \bG_m^n$. By a monomial self-map of $\bP^n$ we mean an endomorphism of $\bG_m^n \subset \bP^n$ given by $\lambda = \vec{0}$ and any matrix $A \in M_{n,n}(\Z)$ viewed as a rational self-map of $\bP^n$.  

Given a variety $X$ and a rational self-map $\Phi$, by a \emph{change of coordinates} we mean an automorphism $\iota: X \lra X$ changing the dynamical system from $(X, \Phi)$ to $(X, \iota \circ \Phi \circ \iota^{-1})$. 

For a rational map $f: X \dra Y$ between varieties $X$ and $Y$ and a subset $Z \subseteq Y$, we abuse notation and let $f^{-1}(Z)$ denote the set-theoretic inverse image of $Z$ under the restriction $f|_{\text{dom}(f)}: \text{dom}(f) \lra Y$ where $\text{dom}(f)$ is the domain of definition for the rational map $f$. 

\subsection{Outline of the paper.} In Section \ref{sec:examples} we give interesting examples of eventually regular self-maps of surfaces. In Section \ref{sec:facts} we collect all the facts related to endomorphisms of smooth projective surfaces that we need to prove our main results. In Section \ref{sec:val-tree} we review the basics of the valuative tree at infinity and also prove an extension of a proposition appearing \cite{FJ07} that we will need in our proofs (See Proposition \ref{prop:5.3-extension}). Section \ref{sec:lemmas} is dedicated to two elementary but useful lemmas which we will use frequently in our proofs. In Section \ref{sec:key-prop} we prove Proposition \ref{prop:totally-inv}. In Section \ref{sec:good-normal-model} we digress a bit to prove the existence of normal models for eventually regular self-maps of varieties. This is mainly inspired by \cite[Proposition 2.12]{favre}. We also describe the structure of the normal model in the case where $X$ is a projective surface. However, we do not need any of the results of this section for the proofs of our main results. In Sections \ref{sec:toric} and \ref{sec:P1-bundle-over-C} we handle the case of toric surfaces and $\bP^1$-bundles over a curve of genus at least 1, respectively. In Section \ref{sec:proof-main-thms} we put the results from Sections \ref{sec:toric} and \ref{sec:P1-bundle-over-C} together to prove Theorems \ref{thm:main} and \ref{thm:more-precise}. Finally, in Section \ref{sec:birational} we describe the structure of totally invariant subvarieties of $\bP^n$ under rational self-maps not preserving non-constant fibrations (See Proposition \ref{prop:inv-subvarieties}) and use it to prove Theorem \ref{thm:birational}.

\subsection*{Acknowledgments.} I would like to express my gratitude to my advisor, Laura DeMarco, for her constant support, advice, and constructive feedback throughout this project. I am also very grateful to Jason Bell, Dragos Ghioca, and Zinovy Reichstein for carefully reading a preliminary version of this paper and offering many valuable comments. I thank Jit Wu Yap for numerous helpful conversations on this topic. I am indebted to Max Weinreich for drawing my attention to \cite{FJ07}, which proved especially useful for this work, and for many valuable comments on an earlier version of the article. I am grateful to Charles Favre for introducing me to \cite[Proposition 2.12]{favre}, which was the main source of inspiration for Section~\ref{sec:good-normal-model}. I owe special thanks to Serge Cantat for providing Example~\ref{ex:prod}, which gives a negative answer to the original formulation of Question~\ref{question:boundIterate}, and for suggesting a refinement of the statement of Theorem~\ref{thm:main}. I also thank Mattias Jonsson for many helpful conversations and for his feedback on Section~\ref{sec:good-normal-model}. Finally, I would like to thank Nathan Chen for his comments on the earlier version of the article and pointing out an error in the statement of Proposition \ref{prop:totally-inv}.

\section{Examples}\label{sec:examples}

Recall that when $X$ is not a toric surface, Theorem \ref{thm:more-precise} shows that any eventually regular and non-invertible endomorphism is already regular. In this section we will give a series of interesting examples of rational self-maps of $\bP^2$ that are not regular themselves but are eventually regular and non-invertible. We also provide a clarifying example regarding the eventually regular and invertible rational self-maps on surfaces with $\kappa(X) = -\infty$ in the end.   
\begin{example}
\label{ex:2-regular}
Consider the rational map $\Phi$ given by $(x,y) \mapsto (x^2, y^{-2})$ on $\C^2$. We can also view $\Phi$ as a rational self-map of $\bP^2$ and is given in homogeneous coordinates by 
\[
[X:Y:Z] \mapsto [X^2Y^2:Z^4:Y^2Z^2]
\]
which has indeterminacy at the points $[0:1:0]$ and $[1:0:0]$. But, $\Phi^2$ is given by 
\[
[X:Y:Z] \mapsto [X^4:Y^4:Z^4],
\]
which is regular everywhere on $\bP^2$. In general, any map of the form $(x, y) \mapsto (p(x), y^{-d})$ for a polynomial of degree $d \ge 1$ will have the property that the second iterate is regular on $\bP^2$. 
\end{example}

\begin{example}
\label{ex:2-reg-mix}
Consider the map $\Phi$ defined by $(x,y) \mapsto (x + y, x^2 + y)$ on $\C^2$. Then, $\Phi:\bP^2 \dra \bP^2$ is given in homogeneous coordinates by 
\[
[X:Y:Z] \mapsto [XZ + YZ: X^2 + YZ : Z^2]. 
\]
This map has a singularity at $[0:1:0]$. But, $\Phi^2$ is given by 
\[
[X:Y:Z] \mapsto [X^2 + XZ + 2YZ: 2X^2 + 2XY + Y^2 + YZ : Z^2],  
\]
which is regular everywhere on $\bP^2$. This example is interesting since, in contrast to Example \ref{ex:2-regular} and the rest of the examples we will see, none of the iterates are a product of the form $(x,y) \mapsto (a(x),b(y))$ for rational maps $a$ and $b$. In other words, the map $\Phi$ ``mixes" the variables $x$ and $y$. 

In general, any map of the form $(x,y) \mapsto (c_1x + c_2y, c_3x^2 + c_4y)$ for $c_1,c_2,c_3,c_4 \in \C$ with $c_1c_2c_3 \ne 0$ has the property that $\Phi^2$ is the first iterate that becomes regular on $\bP^2$. Identifying the self-maps that are the same after conjugation by automorphisms of $\bP^2$, this gives a two-dimensional family of rational self-maps of $\bP^2$ with this property. Indeed, if we let $c_2 = c_3 = 1$, then it is not difficult to check that the maps corresponding to $(c_1,1,1,c_4)$ are all distinct modulo conjugation by an affine transformation of $\C^2$. 
\end{example}
\begin{example}
\label{ex:12-reg}
Let $\Phi$ be the rational self-map of $\C^2$ given by 
\[
(x,y) \mapsto (x^3y, x^{-3}y^{3}). 
\]
This map can be viewed as an endomorphism of $\bG_m^2 \simeq \C^2 \setminus\{xy = 0\}$ which is given by the two by two matrix $$A = \begin{pmatrix}
    3 & 1 \\ -3 & 3
\end{pmatrix}.$$ It is not difficult to see that for $\Phi^k$ to extend to $\bP^2$, $A^k$ must be a diagonal matrix with positive entries. The first power of $A$ that has this property is $$A^{12} = \begin{pmatrix}
    2985984 & 0 \\ 0 & 2985984
\end{pmatrix}.$$ So, $\Phi^{12}$ is the first iterate that extends to an endomorphism of $\bP^2$. 

\end{example}
\begin{example}
\label{ex:3,4,6,8}
 Take $\Phi_1,\Phi_2, \Phi_3$, $\Phi_4$, and $\Phi_5$ to be endomorphisms of $\bG_m^2$ given by matrices 
\[ A_1 = \begin{pmatrix}
    -2 & 0 \\ 0 & -2
\end{pmatrix},  A_2 = \begin{pmatrix}
    -2 & -2 \\ 2 & 0
\end{pmatrix}, A_3 = \begin{pmatrix}
    0 & -2 \\ 2 & 0
\end{pmatrix}, A_4 = \begin{pmatrix}
    2 & 2 \\ -2 & 0
\end{pmatrix} \text{ and } A_5 = \begin{pmatrix}
    1 & 1 \\ -1 & 1
\end{pmatrix},
\]
respectively. Then, arguing as in Example \ref{ex:12-reg}, the first iterates of $\Phi_1,\Phi_2,\Phi_3,\Phi_4,$ and $\Phi_5$ that extend to $\bP^2$ are $\Phi_1^2$, $\Phi_2^3$, $\Phi_3^4, \Phi_4^6$, and $\Phi_5^8$. 
\end{example}
\begin{remark}
\label{rem:bad-then-Gm}
One important thing to note is, as can already be seen from the above examples in the case of $\bP^2$, the rational self-maps of toric surfaces that take more than two iteration to become regular come from endomorphisms of some open subvariety $U \simeq \bG_m^2$ of the toric surface. This will be a consequence of our proofs of Theorem \ref{thm:P2} and Propositions \ref{prop:P1-times-P1} and \ref{prop:Hirzebruch} (See Remarks \ref{rem:long-then-Gm-P2} and \ref{rem:long-then-Gm-P1-bundle}). 
\end{remark}
\begin{remark}
\label{rem:examples-for-second-part}
Examples \ref{ex:12-reg} and \ref{ex:3,4,6,8} along with the monomial endomorphism of $\bP^2$ given in affine coordinates by $(x,y) \mapsto (x^2,y^2)$ prove the second part of Theorem \ref{thm:main}.   

\end{remark}
\begin{example}
\label{ex:bir-large-order}
Let $\zeta_n$ be a primitive $2n$-th root of unity for $n \ge 3$ and consider the automorphism $\Phi_n: \A^2 \lra \A^2$ given by 
\[
(x,y) \mapsto (c\zeta_n x, c^2y + x^2). 
\]
for some $c \in \C^\ast$. Then, it is easy to see that $\Phi_n^n = (c^n\zeta_n^{n}x, c^{2n}y)$ and thus it induces a regular map on $\bP^2$ while all of $\Phi_n, \cdots, \Phi_n^{n-1}$ are not regular as rational self-maps on $\bP^2$. Note that $\Phi_n^n$ preserves the fibration $f: \A^2 \lra \A^1$ given by $f(x,y) = \frac{x^2}{y}$. 
\end{example}
\begin{example}
\label{ex:prod}
Fix $n \ge 3$ and let $\Phi_n: \bP^2 \dra \bP^2$ be the rational map defined in Example \ref{ex:bir-large-order}. Let $X$ be any variety over $\C$ admitting a non-invertible endomorphism $\Psi: X \lra X$. Then, the rational self-map $(\Psi, \Phi_n)$ on $X \times \bP^2$ becomes regular after exactly $n$ iterations. 
\end{example}
\begin{example}
\label{ex:non-minimal}
Let $X = E \times E$ where $E$ is an elliptic curve. By the discussion in \cite[Section 2.4.6]{cantat-aut}), $X$ admits loxodromic automorphisms, namely, those given by matrices in $\text{SL}_2(\Z)$ with trace greater than $2$.  Let $f$ be a loxodromic automorphism of $X$ and let $x \in X$ be a periodic point of period $n \ge 1$ under $f$. Then, $f$ induces a rational self-map of the blowup $\text{Bl}_x(X)$ of $X$ at $x$. Since $x$ takes exactly $n$ iterations under $f$ to return to itself, we conclude that $f^n$ is the first iterate that becomes regular on $\text{Bl}_x(X)$. Note that the blowups $\text{Bl}_x(X)$ for different $x \in X$ are all isomorphic, say, to a smooth surface $\hat{X}$, since the action of $\text{Aut}(X)$ on $X$ is transitive. Also, $n$ could be chosen to be arbitrarily large by \cite[Theorem 4.7]{cantat-aut}. Moreover, since $f$ is loxodromic, it cannot preserve a non-constant fibration by \cite[Lemma 2.5]{ML24}.

This example shows that there cannot be a bound on the number of iterations it takes for an eventually regular birational self-map $\Phi$ of $\hat{X}$ to become regular even under the extra assumption that $\Phi$ does not preserve a non-constant fibration.
\end{example}

\section{Facts about endomorphisms of smooth projective surfaces}
\label{sec:facts}
Our first step in proving Theorem \ref{thm:main} is to understand which smooth projective surfaces can have non-invertible endomorphisms. The next two results describe such surfaces and will be an important ingredient in our proof of Theorem \ref{thm:main}. 

\begin{theorem}
\label{thm:pos-kod}
Suppose $X$ is a surface with $\kappa(X) \ge 0$ that admits a non-invertible surjective endomorphism $f$. Then, $X$ is minimal and $f$ is a finite \'etale morphism. 
\end{theorem}
\begin{proof}
See \cite[Lemma 2.3 and Theorem 3.2]{Fujimoto}.    
\end{proof}
In the case of non-invertible surjective endomorphisms of ruled surfaces we have the next classification theorem.
\begin{theorem}
\label{thm:ruled}
Let $X$ be a ruled surface. It has a non-invertible surjective endomorphism if and only if $X$ is one of following surfaces:
\begin{enumerate}
    \item a toric surface;
    \item a $\bP^1$-bundle over an elliptic curve;

    \item a $\bP^1$-bundle over a non-singular projective curve $B$ of genus $g(B) > 1$ such that
    $X \times_B B' \simeq \bP^1 \times  B'$ for some finite \'etale morphism $B' \lra B$. 
\end{enumerate}
\end{theorem}
\begin{proof}
See \cite[Theorem 3]{Nak02}.
\end{proof}

In the rest of this section we gather some facts about non-invertible endomorphisms of $\bP^1$-bundles and $\bP^2$. We quickly recall that a $\bP^1$-bundle over an algebraic curve $C$ is given by $\bP(E)$ for a vector bundle $E$ of rank 2 over $C$. There is a natural projection map $\pi: \bP(E) \lra C$ whose fibers are isomorphic to $\bP^1$. Now, let $X$ be a $\bP^1$ bundle over a smooth curve $C$ with projection map let $\pi: X \lra C$. We have the next lemma which limits the possibilities of non-invertible surjective endomorphism of $\bP^1$-bundles. 
\begin{lemma}
\label{lem:MSS-5.8}
Let $X$ be as above and $f$ be a non-invertible surjective endomorphism on $X$. Then, $f^2$ must preserve the fibers of $\pi$. Moreover, if $X$ is a Hirzebruch surface $\F_n$ for some $n \ge 1$, then $f$ must preserve the fibers of $\pi$. 
\end{lemma}  
\begin{proof}
See \cite[Lemma 5.4 and Lemma 5.8]{MSS}. 
\end{proof}
By Lemma \ref{lem:MSS-5.8}, at the expense of replacing a non-invertible endomorphism $f: X \lra X$ with $f^2$, we can assume that there exists an endomorphism $g: C \lra C$ such that $g \circ \pi = \pi \circ f$. We will sometimes refer to $g$ as \emph{the endomorphism induced on $C$ by $f$}. When $X$ is $\F_n$ for some $n \ge 1$, we have the next useful lemma which gives us a handle on the degree of the induced map $g$.
\begin{lemma}
\label{lem:MSS-5.10}
Let $X$ be a Hirzebruch surface $\F_n$ for some $n \ge 1$ and $f$ be a non-invertible endomorphism of $X$. If $g$ is the induced endomorphism on $C$ by $f$ and $F$ is any fiber of the projection $\pi$, then $\deg(g) = \deg(f|_F)$. In particular, if $g$ is an automorphism then so is $f$.  
\end{lemma}
\begin{proof}
For the proof of the first conclusion see \cite[Lemma 5.10]{MSS}. The second conclusion is clear since $\deg(f) = \deg(g)\deg(f|_F) = \deg(g)^2$. 
\end{proof}
The next lemma is a strengthening of Lemma \ref{lem:MSS-5.8} in the case where $C$ has genus at least 1 and $f$ is a rational self-map (not necessarily regular). 
\begin{lemma}
\label{lem:rat-pres-fibers}
Let $X$ be a $\bP^1$-bundle over a smooth projective curve $C$ over $\C$ of genus at least 1. Let $\pi: X \lra C$ be the projection map. Suppose $f$ is a dominant rational self-map of $X$. Then, there must exist an endomorphism $g: C \lra C$ such that $g \circ \pi = \pi \circ f$. 
\end{lemma}
\begin{proof}
Let $F \simeq \bP^1$ be a fiber of the projection $\pi$. Then, $\pi(f(F))$ is either equal to $C$ or just a point. $\pi(f(F)) = C$ would be a contradiction since $C$ has genus at least 1 and there are no morphisms from $\bP^1$ to $C$ by the Riemann-Hurwitz formula. Hence, $\pi(f(F))$ must be a point. This implies that every fiber of the projection $\pi$ is contracted to a point by $\pi \circ f$. Therefore, $\pi \circ f$ must factor through the projection $\pi$ as desired.  
\end{proof}

The next result is very useful to us in our proof of Theorem \ref{thm:more-precise} for a non-minimal toric surface $X$.

\begin{lemma}
\label{lem:S(X)-tot-inv}
Let $X$ be a smooth projective surface over $\C$ admitting a non-invertible endomorphism $f$. Let $S(X)$ denote the set of curves in $X$ with negative self-intersection. Then, $S(X)$ is a finite set and $f^{-m}(C) = C$ for some $m \ge 1$ and every $C \in S(X)$. 
\end{lemma}
\begin{proof}
By \cite[Proposition 10]{Nak02}, $S(X)$ is a finite set and there exists an integer $m \ge 1$ such that $f^{m}$ fixes all of the curves in $S(X)$. Moreover, by the proof of \cite[Lemma 8]{Nak02} for any $C \in S(X)$ and any $C'$ such that $f^{m}(C') = C$ we must have $C' = C$. In other words, we must have $f^{-m}(C) = C$. 
\end{proof}

The next theorem classifies the totally invariant hypersurfaces of $\bP^2$ and is crucial in our proof of Theorem \ref{thm:more-precise} in the case of $X = \bP^2$. 
\begin{theorem}
\label{conj:lin-conj}
Let $\Phi$ be a non-invertible endomorphism of $\bP^2$. Then, all totally invariant curves of $\bP^2$ under $\Phi$ are linear. 
\end{theorem}
\begin{proof}
The proof follows from theorem \cite[Theorem 1.1]{Hor17} and the discussion after the statement of the theorem. 
\end{proof}
\section{Preliminaries and facts about the valuative tree at infinity}
\label{sec:val-tree}
In this section, we review fundamental facts about the valuative tree at infinity. In the end, we prove a result related to periodic points of the action induced on the valuative tree by a polynomial endomorphism of $\C^2$ of maximal entropy (See Proposition \ref{prop:5.3-extension}). For a more detailed treatment of the basic facts about the valuative tree, we refer the reader to \cite{eigenvaluations}, \cite{FJ07}, and the accessible notes \cite{Jon11}. We begin by quickly recalling some basic definitions about rooted $\R$-trees. 
\subsection{Rooted \texorpdfstring{$\R$}{R}-trees}\label{sec:rooted-trees} We recall the next definition from \cite[Definition 2.4]{Jon11}.
\begin{definition}
\label{def:rooted-tree}
A rooted tree is a partially ordered set $(X,\le)$ satisfying the following properties:
\begin{itemize}
    \item[(RT1)] $X$ has a unique minimal element $x_0$;
    \item [(RT2)] For any $x \in X\setminus\{x_0\} $, the set $ \{z \in X : z \le x\}$ is isomorphic (as a partially
ordered set) to the real interval $[0, 1] $; 

    \item[(RT3)] any two points $x,y \in X$ admit an infimum $x\wedge y$ in $X$, that is, $z \le x$ and $z \le y$ if and only if $z \le x \wedge y$;

    \item[(RT4)] any totally ordered subset of $X$ has a least upper bound in $X$.
\end{itemize}

Given two points $x,y$ in a rooted $\R$-tree $X$ with $x \le y$ we define 
\[
[x,y] = \{z \in X: x \le z \le y\}. 
\] 
For any two points $x,y \in X$ we can then define
\[
[x,y] = [x \wedge y, x] \cup [x \wedge y, y]. 
\]
\end{definition} 

For any point $x \in X$ we can define an equivalence relation on $X$ as follows: Say $y \sim z$ whenever $[x,y]$ and $[x,z]$ have an intersection point other than $x$. A \emph{tangent direction} $\vec{v}$ at $x$ is then an equivalence class of this relation and we let $U(\vec{v})$ denote the set of all elements in the equivalence class of $\vec{v}$. We let $Tx$ denote the set of all tangent directions at $x$.

Lastly, we note that there are two important types of points in a (rooted) tree $X$: \emph{Ends} which are points with a unique tangent direction and \emph{Branch points} which are points with at least three distinct tangent direction.

\subsection{Admissible compactifications} In what comes next we fix an embedding of $\C^2$ in $\bP^2$ and define $L_\infty$ to be $\bP^2 \setminus\C^2$. An \emph{admissible compactifiction} of $\C^2$ is a smooth projective surface $X$ admitting a birational morphism $\pi: X \lra \bP^2$ that is an isomorphism above $\C^2$. 

\subsection{Valuations} Let $R$ be the coordinate ring of $\C^2$. A \emph{valuation} on $R$ is a function 
\[
v : R \to \mathbb{R} \cup \{+\infty\},
\]
satisfying:
\begin{enumerate}
    \item \(v(PQ)=v(P)+v(Q)\) for all \(P,Q\in R\),
    \item \(v(P+Q) \ge \min\{v(P),v(Q)\}\) for all \(P,Q\in R\),
    \item \(v(\lambda)=0\) for all \(\lambda\in\mathbb{C}^*\).
\end{enumerate}

We denote by $\widehat{\V}_\infty$ the space of all valuations $\nu$ on $R$ such that $\nu(L) < 0$ for a generic affine function $L$ on $\C^2$. The \emph{valuative tree at infinity} is then defined to be the subset of all valuations $\nu \in \widehat{\V}_\infty$ normalized so that $\nu(L) = -1$ for a generic affine function $L \in C^2$. 

\subsection{The tree structure and topology on \texorpdfstring{$\V_\infty$}{Vinfty}} There is a natural partial ordering on $\V_\infty$ given by 
\[
\nu \le \mu \Longleftrightarrow \nu(P) \le \mu(P) \text{ for all } P \in R.
\]
which turns $\V_\infty$ into a tree. The tree has a unique minimal element corresponding to the valuation $-\deg := \ord_{L_\infty}$ which allows us to also view $\V_\infty$ as a rooted tree with root $-\deg$. Given any admissible compactifiction $X$ and any prime divisor $E \subseteq X \setminus\C^2$, we can define a normalized valuation 
\[
\nu_E := b_E^{-1}\ord_E. 
\]
where $b_E = \ord_E(L)$ for a generic affine function $L$. Any such valuation is called a \emph{divisorial} valuation. 

A special type of divisorial valuation is a \emph{pencil valuation} which is defined as follows. Let $C$ be an affine curve in $\C^2$ given by $\{P = 0\}$ for some $P \in R$. Assume that the closure of $C$ in $\bP^2$ is irreducible and has one intersection point with the line at infinity. Let $C_\lambda = \{P = \lambda\}$. Then, we can define 
\[
\nu_{|C|}(Q) = \deg(C)^{-1}\ord_{\infty}(Q|_{C_\lambda}),
\]
for every $Q\in R$. This is called the pencil valuation corresponding to the pencil $|C|$. If $C$ is assumed to be a rational curve we call such a valuation a rational pencil valuation. See \cite{eigenvaluations} for further details on the different types of valuations in $\V_\infty$.

One can endow $\V_\infty$ with a natural weak topology where $\nu_n \to \nu$ whenever $\nu_n(P) \lra \nu(P)$ for all $P \in R$. $\V_\infty$ is compact with respect to this topology (See \cite[Appendix A]{eigenvaluations}).   

\subsection{The main two parametrizations of \texorpdfstring{$\V_\infty$}{Vinfty}}
For a rooted tree $(T, \le)$ we recall that a \emph{parametrization} is a monotone function $\alpha: T \lra [-\infty,\infty]$ such that the restriction of $\alpha$ to any interval $[x,y]$ for some $x < y$ is a homeomorphism onto a closed sub-interval of $[-\infty,\infty]$. Moreover, $|\alpha(x_0)| \ne \infty$ unless $x_0$ is an end of the tree. Given a parametrization $\alpha$ one can define a metric on $T$ by 
\[
d(x,y) = |\alpha(x) - \alpha(x \wedge y)| + |\alpha(y) - \alpha(x \wedge y)|, 
\]
where $x \wedge y$ denotes the unique maximal element $z$ of $T$ such that $z \le x$ and $z \le y$ (See Definition \ref{def:rooted-tree}). 

There exists a unique parametrization $\alpha: \V_\infty \lra [-\infty, 1]$ called \emph{skewness} satisfying the following properties:
\begin{itemize}
    \item[(a)] $\alpha$ is decreasing and upper semi-continuous

    \item[(b)] $\alpha(-\deg) = 1$ and $|\alpha(\nu_E) - \alpha(\nu_{E'})| = (b_Eb_{E'})^{-1}$ whenever $E$ and $E'$ are intersecting prime divisors in some admissible compactification of $\C^2$.  
\end{itemize}
Using the discussion above, we can use the parametrization $\alpha$ to give $\V_\infty$ the structure of rooted a metric tree. 

Similar to $\alpha$, there is a unique increasing and lower semi-continuous parametrization $A: \V_\infty \lra [-2, \infty]$ called \emph{thinness} satisfying $A(-\deg) = -2$ and $A(\nu_E) = \frac{a_E}{b_E}$ where $a_E = 1 + \ord_E(dx \wedge dy)$. 

Using the two parametrization $\alpha$ and $A$ we can define a subtree $\V_1$ of $\V_\infty$ consisting of valuation $\nu$ with $\alpha(\nu) \ge 0$ and $A(\nu) \le 0$. We will refer to $\V_1$ as the \emph{tight tree}. $\V_1$ is of significant importance since it is preserved under the action of polynomial endomorphisms of $\C^2$ (See Proposition \ref{prop:induced-map-on-V}). 
\subsection{Induced map of polynomial endomorphisms on \texorpdfstring{$\V_\infty$}{Vinfty}} Let $F: \C^2 \lra \C^2$ be a polynomial endomorphism and $F^\ast: R \lra R$ be the pull-back map sending $a \in R$ to $a \circ F \in R$. For any $\nu \in \widehat{V}_\infty$ we define $d(F, \nu) := \nu(F^\ast(L))$ for a generic affine function $L$. We can then define a map $F_\ast$ on $\widehat{V}_\infty$ by
\[
F_\ast(\nu) = \begin{cases}
    P \mapsto \nu(F^\ast(P)) \text{ \: if } d(F,\nu) > 0, \\
    P \mapsto 0 \text{ \: if } d(F,\nu) = 0.
\end{cases}
\]
for any $\nu \in \widehat{V}_\infty$.  For any valuation $\nu \in \V_\infty$ with $d(F,\nu) > 0$ we define $F_\bullet(\nu) = \frac{F_\ast(\nu)}{d(F,\nu)}$. When $F$ is not a proper map, $F_\ast$ could send a valuation in $\V_\infty$ to a valuation outside of $\V_\infty$. So, $F_\bullet$ is only well-defined on an open subset of $\V_\infty$. The next proposition shows that if we restrict our attention to the tight tree $\V_1$, then $F_\bullet$ induces a "nice" endomorphism of $\V_1$.

\begin{proposition}
\label{prop:induced-map-on-V}
The map $F_\bullet$ is well-defined and continuous self-map on the tight tree $\V_1$ and it preserves the tree structure on $\V_1$. Moreover, if $F$ is a proper map, then $F_\bullet$ is well-defined, continuous and surjective as an endomorphism of $\V_\infty$ and it preserves the tree structure. 
\end{proposition}
\begin{proof}
See \cite[Section 7.2]{eigenvaluations}. 
\end{proof}

\begin{definition}[$\lambda_1$ and $\lambda_2$]
\label{def:dynamical-degrees}
For a dominant polynomial endomorphism $F: \C^2 \lra \C^2$ we define 
\[
\lambda_1(F) := \lim_{n \to \infty} \deg(F^n)^{1/n} = \lim_{n \to \infty} d(F^n, -\deg)^{1/n}. 
\]
$\lambda_1(F)$ is known as the \emph{asymptotic degree} or \emph{the first dynamical degree} of $F$. 

For a dominant polynomial endomorphism $F: \C^2 \lra \C^2$ we define $\lambda_2(F)$ to be the topological degree of $F$ that is the number of distinct elements in the generic fiber of $F$. 
\end{definition}

The next proposition shows that the map $F_\bullet: \V_1 \lra \V_1$ will always have a fixed point.  
\begin{proposition}
\label{prop:eigenvaluation}
There exists a valuation $\nu_\ast \in \V_1$ such that $F_\ast(\nu_\ast) = \lambda_1\nu_\ast$. Any such valuation will be called an \emph{eigenvaluation}.
\end{proposition}
\begin{proof}
See \cite[Proposition 2.3]{FJ07} and \cite[Section 7.3]{eigenvaluations}.
\end{proof}

\subsection{Dynamical facts in the case of maximum topological degree} In this subsection we assume that $\lambda_1(F)^2 = \lambda_2(F)$ and $\lambda_1(F) > 1$. Let 
\[
\mathcal{T}_F = \{\nu \in \V_1 : F_\bullet (\nu) = \nu\}. 
\]
The next proposition gives a precise description of the structure of $\mathcal{T}_F$ in the case where $\deg(F^n)/\lambda_1^n$ remains bounded. 
\begin{proposition}
\label{prop:FJ-5.3}
Suppose that $\deg(F^n)/\lambda_1^n$ is bounded for $n \ge 1$. Then, $\mathcal{T}_F$ is non-empty and all of its elements are totally invariant under $F_\bullet$. Moreoever, the following must hold:
\begin{enumerate}
    \item The set $\mathcal{T}_F$ is either a singleton consisting of a quasimonomial valuation $\nu$ with $\alpha(\nu) > 0$, or it is an interval with divisorial ends.

    \item If $\mathcal{T}_F$ is an interval, then $\mathcal{T}_F = \mathcal{T}_{F^n}$ for all $n \ge 1$.

    \item If $\mathcal{T}_F$ is a singleton, then $\mathcal{T}_{F^2}$ is an interval with $\mathcal{T}_F$ in its interior. Moreover, $\mathcal{T}_{F^2} = \mathcal{T}_{F^{2n}}$ for all $n \ge 1$. 
\end{enumerate}
\end{proposition}
\begin{proof}
See \cite[Proposition 5.3 and Lemma 5.4]{FJ07}
\end{proof}

The next result does not appear in \cite{FJ07} but can be shown using similar arguments as in the proof of proposition \cite[Proposition 5.3(b)]{FJ07}. We will need it later when proving Theorem \ref{thm:more-precise} in the case of a general toric surface. 
\begin{proposition}
\label{prop:5.3-extension}
Suppose we are in the situation of Proposition \ref{prop:FJ-5.3}. Let $T$ be a finite subtree of $\V_\infty$ which intersects $\T_f$ in an interval $I$. Assume that all elements of $T$ are totally invariant under an iterate $F_\bullet^\ell$ for some $\ell \ge 1$. Then, $T$ is an interval and all elements of $T$ are totally invariant under $F_\bullet$. 
\end{proposition}
\begin{proof}
We first prove that $T$ is a segment. To do this it suffices to show that $T$ has no branch points (See section \ref{sec:rooted-trees} for the definition of branch points). So, suppose for the sake of contradiction that $\nu \in T$ is a branch point. Then $\nu$ must be a divisorial valuation corresponding to an exceptional divisor $E$ (See \cite[Section 1.6]{eigenvaluations}). Since $T$ is totally invariant under an iterate $F^\ell_\bullet$, we have $F_\bullet^\ell(\nu) = \nu$. In particular, since $F_\bullet$ preserves the tree structure on $\V_\infty$ we get an induced tangent map $F_\nu^\ell: T\nu \lra T\nu$. But recall that the tangent directions at $\nu$ correspond to the points of $E \simeq\bP^1$ and the tangent map $F^\ell_\nu$ corresponds to a rational map on $E \simeq\bP^1$ induced by $F^\ell$. So, at least three of the points on $E$ are totally invariant under $F^\ell$ and we conclude that all points are totally invariant because of this. In other words, $F$ must be the identity on $E$ which is a contradiction (We invite the reader to read the proof of \cite[Proposition 5.3]{FJ07} for a similar type of argument).  

 Hence, $T$ is an interval in $\V_\infty$ with endpoints that we will call $\nu_1$ and $\nu_2$. If $\nu_1, \nu_2 \in \V_1$, then $T \subseteq \V_1$ and we conclude by Proposition \ref{prop:FJ-5.3} that $T \subset \T_f$ all the elements of $T$ are already totally invariant under $\Psi$ which finishes the proof. So, we assume that at least one of $\nu_1$ and $\nu_2$ is outside of $\V_1$. 

 Suppose $\nu_1 \notin \V_1$ and let $I =\T_f \cap T = [\nu_3,\nu_4]$ for  $\nu_3,\nu_4 \in \V_1$. By  Proposition \ref{prop:FJ-5.3} we know that all elements of $I$ are totally invariant under $F_\bullet$. Also, $\nu_1$ must be larger than one of $\nu_3$ and $\nu_4$ with respect to the natural partial ordering on $\V_\infty$. Assume $\nu_1 > \nu_3$ without loss of generality and define $J := [\nu_3,\nu_1]$. Our goal is to show that all elements of $J$ are totally invariant under $F_\bullet$. Let $\nu_\ast$ be the maximal element in $J$ such that all the elements in the interval $[\nu_3, \nu_\ast]$ are totally invariant under $F_\bullet$. If $\nu_\ast = \nu_1$ we are done. So, assume that $\nu_\ast < \nu_1$. The tangent direction at $\nu_\ast$ containing $\nu_3$ is totally invariant under $F_\bullet$ and the tangent direction containing $\nu_1$ is totally invariant under an iterate of $F_\bullet$. We conclude that the tangent containing $\nu_1$ is totally invariant under $F_\bullet$ itself because otherwise we get at least three totally invariant tangent directions at $\nu_\ast$ and which yields a contradiction as before.   

 Choose an interval $I' = [\nu_\ast,\nu]$ with $\nu$ sufficiently close to $\nu_\ast$ so that all $I',F_\bullet(I'),\dots,F_\bullet^{\ell-1}(I')$ are intervals contained in $J$ and such that $F_\bullet$ restricted to $I'$ is a homeomorphism onto its image (possible since $F_\bullet$ is a regular tree map). If $I',F_\bullet(I'),\dots,F_\bullet^{\ell-1}(I')$ form an ascending chain 
 \[
 I'\subsetneq F_\bullet(I') \subsetneq \cdots \subsetneq F_\bullet^{\ell-1}(I')
 \]
 we get 
 \[
 F_\bullet(I') \subsetneq \cdots \subsetneq F_\bullet^{\ell-1}(I') \subsetneq F^\ell_\bullet(I') = I'
 \]
by applying $F_\bullet$ to the original chain which is absurd. Hence, $I'' := F_\bullet^s(I') \supset F_\bullet^{s+1}(I') = F_\bullet(I'')$ for some $0 \le s \le \ell - 2$. Then, by \cite[Theorem 4.5]{eigenvaluations} $F_\bullet$ must have a fixed point on $I''$ which is either in the interior or is an attracting endpoint. Since $F_\bullet^\ell$ induces the identity on $I''$ there can be no attracting endpoints in $I''$. So, there is a point $\nu_1 \in (I'')^o$ such that $F_\bullet(\nu_1) = \nu_1$. 

Therefore, $F_\bullet$ sends the interval $[\nu_\ast,\nu_1]$ to itself. Using \cite[Theorem 4.5]{eigenvaluations} and arguing as above we can prove that given any two fixed points $\nu'$ and $\nu''$ in $[\nu_\ast,\nu]$ there is a point in between $\nu'$ and $\nu''$ that is also fixed. Hence, the set of fixed points in $[\nu_\ast,\nu]$ is dense. It follows that all points in $I''$ are fixed by continuity of $F_\bullet$. Therefore,  all elements in the interval $[\nu_3, \nu]$ are totally invariant under $F_\bullet$ since they are fixed under $F_\bullet$ and because they are totally invariant under $F_\bullet^\ell$. Here we also use the fact that $F_\bullet$ maps $\V_\infty$ onto itself by Proposition \ref{prop:induced-map-on-V}. This contradicts the maximality of $\nu_\ast$. We conclude that all elements of $J$ are totally invariant under $F_\bullet$. 

Similarly, if $\nu_2 \notin \V_1$ we can show that the elements of the interval $[\nu_4,\nu_2]$ are totally invariant under $F_\bullet$. Therefore, the elements of $T = [\nu_1,\nu_2] = [\nu_4,\nu_2] \cup [\nu_3,\nu_4] \cup [\nu_1,\nu_3]$ are totally invariant under $F_\bullet$ and we are done. 
\end{proof}
\section{Lemmas}
\label{sec:lemmas}
In this section we collect lemmas that will be used in our proof of Theorem \ref{thm:more-precise}. We start with the next simple lemma which allows us to bound the smallest power of a matrix that becomes diagonal when we know that some power of the matrix is diagonal. This lemma will be key in the proof of Theorem \ref{thm:more-precise} when the rational self-map $\Phi$ is assumed to induce an endomorphism of an open subvariety of $X$ that is isomorphic to $\bG_m^2$. 
\begin{lemma}
\label{lem:smallestPower}
Let $A \in M_{2,2}(\Z)$ be a two-by-two matrix. Suppose that $A^k$ is a diagonal matrix. Then, $A^i$ is a diagonal matrix for some $i\in\{1,2,3,4,6\}$. Consequently, there exists $i\in \iterates$ such that $A^i$ is diagonal with non-negative entries.
\end{lemma}
\begin{proof}
The second part of the lemma is clearly a consequence of the first part. To prove the first part, suppose that $A^k = \begin{pmatrix}
    d & 0 \\ 0 & e
\end{pmatrix}$ for some $d, e \in \Z$ and let $\lambda$ and $\mu$ be the eigenvalues of $A$. Both $\lambda$ and $\mu$ must be quadratic integers since $A$ is a 2-by-2 matrix with integer coefficients. Moreover, we must have $\lambda^k = d$ and $\mu^k = e$. So, if $|d| \ne |e|$ we conclude that $\lambda$ and $\mu$ are not Galois conjugates. Hence, they both have to be integers and $A = QBQ^{-1}$ where $B = \begin{pmatrix}
\lambda & 0 \\ 0 & \mu  
\end{pmatrix}$ and $Q \in \text{GL}_{2}(\Q)$. Since $QB^kQ^{-1} = A^k = B^k$ it follows that $B^k$ commutes with $Q$. But, since $|d| \ne |e|$ it is not difficult to see that $B$ itself must also commute with $Q$. Hence, $A = B$ and we can take $i = 1$. So, we can assume $|d| = |e|$. 

After replacing $k$ with $2k$ if necessary, we have $A^k = d \cdot \Id$ for some $d \ge 1$. If $\lambda$ and $\mu$ are not Galois conjugates, then they are both integers. So, $A^2$ is conjugate to $d'\cdot \Id$ for some $d' \ge 1$. It follows that $A^2 = d'\cdot \Id$ itself and we are done. 

We now assume that $\lambda$ and $\mu$ are Galois conjugates. Let $K = \Q(\lambda)$ and let $\sigma$ be the generator of $\text{Gal}(K/\Q)$. Then, $\mu = \sigma(\lambda)$ and
\[
e = \mu^k = \sigma(\lambda^k) = \sigma(d) =  d.
\]
Therefore, $e = d$ and we must have $(\frac{\lambda}{\mu})^k = 1$. In other words, $\frac{\lambda}{\mu}$ is a root of unity. But, $\frac{\lambda}{\mu} \in K$ which means that it must be a quadratic integer. It follows that $\frac{\lambda}{\mu} = \zeta$ for a primitive $i$-th root of unity $\zeta$ with $i \in \{1,2,3,4,6\}$. Therefore, we have $\lambda^i = \mu^i = \sigma(\lambda)^i$. This means that both $\lambda^i$ and $\mu^i$ are equal to an integer $d'$ since they are invariant under the generator of $\text{Gal}(K/\Q)$. So, $A^i = d'\cdot \Id$ which finishes the proof. 
\end{proof}

Often in our proofs we find ourselves in the situation where $\Phi$ is given by a skew product on some open subset $U$ of $X$ and we would like to conclude that $\Phi$ must already be regular on $U$ if some iterate of it is regular on $U$. The next lemma allows us to do this.  
\begin{lemma}
\label{lem:skew-then-split}
Let $C$ be an analytic curve, let $\C(C)$ and $\C[C]$ denote the rings of meromorphic functions and holomorphic functions on $C$, respectively, and let $\varphi(t)$ be an endomorphism of $C$. Suppose $F$ is a self-map of $C \times \A^1$ given by
\[
(x,y) \mapsto (\varphi(x), f(x,y)),
\]
where $f \in \C(C)[y]$ and $F^k$ is given by
\[
(x,y) \mapsto (\varphi^k(x), g(y)),
\]
for some $g \in K[y]$ and some $k \ge 1$ and some ring $K \subseteq \C(C)$. Moreover assume that the leading coefficient of $g$ lies in $K^\ast$. Suppose one of the following holds 
\begin{itemize}
    \item[(a)] $K = \C$ or $\C[C]$ and $\varphi^k$ has finite order and $\deg(g) \ge 2$, or

    \item[(b)] $K = \C$, $C = \A^1$, $\varphi \in \C[x]$ and $f \in \C[x,y]$, or
    \item[(c)] $K = \C$ or $\C[x]$, $C = \A^1$, $\varphi(x) = \alpha x^\ell$ for some $\alpha \in \C^\ast$, $\ell \ge 1$, and $f \in \C[x,x^{-1}, y]$.
\end{itemize}
Then, $f \in K[y]$ with the leading coefficient of $f$ being in $K^\ast$. 
\end{lemma}
\begin{proof}
Note that it suffices to prove the lemma for any multiple of $k$. Therefore, we may assume without loss of generality that if $(a)$ holds then $\varphi^k = Id$. 

Let $f(x,y) = \sum_{i=0}^d a_i(x)y^i$ where $a_0(x), \dots, a_d(x) \in \C(C)$, $a_d(x) \ne 0$ and $d \ge 1$. Suppose that $F^n$ is given by 
\[
(x,y) \mapsto (\varphi^n(x), f_n(x,y))
\]
for polynomials $f_n \in \C(C)[y]$. It follows that $\deg_y(f_n) = d^n$. So, we can write
\[
f_n(x,y) = \sum_{i = 0}^{d^n} a_{i,n}(x)y^i,
\]
for functions $a_{i,n} \in \C(C)$. Moreover, we have $d^k = \deg(g)$. 

We claim that no matter which case we are in $a_d(x)$ must be in $K^\ast$. By a straightforward calculation we know 
\begin{equation}
\label{eqn:LeadingCoeff}
c(x) := a_{d^k,k}(x) = a_d(x)^{d^{k-1}} a_d(\varphi(x))^{d^{k-2}} \cdots a_d(\varphi^{k-1}(x)).    
\end{equation}
By hypothesis we know that $a_{d^k,k}(x)$ is equal to some $c \in K^\ast$. If we are in case (b) then it is clear from \eqref{eqn:LeadingCoeff} that $a_d$ must be in $\C^\ast$ since $a_d(\varphi^i(x))$ is a non-zero polynomials in $\C[x]$ for all $0 \le i \le k-1$. Similarly, in case (c), $a_d$ can only have poles at $0$. Then, using equation \eqref{eqn:LeadingCoeff} and the fact that $\varphi(x) = \alpha x^\ell$ we see that $a_d$ cannot have any poles and thus lies in $\C[x]$. Since $a_{d^k,k} \in K^\ast$, we see from equation \eqref{eqn:LeadingCoeff} that $a_d \in K^\ast$. Lastly, if $K = \C$ in this case, it is clear by equation \eqref{eqn:LeadingCoeff} that $a_d$ must be in $\C^\ast$. If we are in case (a), plugging in $\varphi(x)$ for $x$ in equation \eqref{eqn:LeadingCoeff} yields
\[
c(\varphi(x)) = a_d(\varphi(x))^{d^{k-1}} \cdots a_d(\varphi^k(x)). 
\]
On the other hand raising equation \eqref{eqn:LeadingCoeff} to the power of $d$ gives us
\[
c(x)^d = a_d(x)^{d^{k}} a_d(\varphi(x))^{d^{k-1}} \cdots a_d(\varphi^{k-1}(x))^d.
\]
Dividing the above two equations by each other we get
\[
\frac{c(x)^d}{c(\varphi(x))} = \frac{a_d(x)^{d^k}}{a_d(\varphi^k(x))} = \frac{a_d(x)^{d^k}}{a_d(x)} = a_d(x)^{d^k - 1}.
\]
Since $c(x), c(\varphi(x)) \in K^\ast$, it follows that $a_d(x)^{d^k - 1} \in K^\ast$. Using $d \ge 2$ and the fact that $K = \C$ or $K = \C[C]$, we conclude that $a_d(x)$ must be also be in $K$. To finish the proof, let $j$ be the largest index such that $a_j(x)$ is not in $K$ (we would be done if no such index exists). Using induction on $n$ we can easily conclude that for every $n\ge 1$, $a_{d^n - d +j,n}$ is the coefficient with the largest index in $f_n$ that is not in $K$. This contradicts the fact that $f_k$ is a polynomial in $K[y]$.   
\end{proof}

\section{Proof of the key proposition}
In this section, we prove Proposition \ref{prop:totally-inv}, a key component in our proof of Theorem \ref{thm:more-precise}. First, we start with the next lemma which is an analog of \cite[Lemma 4.1]{BGR}. 
\label{sec:key-prop}
\begin{lemma}
\label{lem:regularComposition}
Let $K$ be an algebraically closed field. Let $X,Y,$ and $Z$ be varieties over $K$, $\Psi: X \lra Y$ be a dominant finite morphism, and $f: Y \dra Z$ and $g:X \dra Z$ be rational maps such that that $f \circ \Psi = g$. Assume that $Z$ is projective. Then, $f$ is regular at $\Psi(x) \in Y$ if and only if $g$ is regular at $x \in X$. 
\end{lemma}
\begin{proof}
If $f$ is regular at $\Psi(x)$, then $g$ must clearly be regular at $x$. Now, assume $g$ is regular at $x \in X$. We need to show that $f$ is regular at $\Psi(x)$. Let $y := \Psi(x)$ and $z = g(x)$. Pick an affine neighborhood $V$ of $y$ and let $U := \Psi^{-1}(V) \cap \text{domain}(g)$. Since $\Psi$ is finite we have $\OO_{X,x} \otimes \kappa(y)$ is finitely generated over $\kappa(y)$ (See \cite[Definition 6.2.2]{EGAII}). Using \cite[Corollary 1.1.8]{JW15} there exist \'etale neighborhoods $u: (U',x') \lra (U,x)$ and $v: (V',y') \lra (V, y)$ such that the lift $\Psi': U' \lra V'$ of $\Psi$ is finite, flat and surjective. By the proof of \cite[Lemma 4.1]{BGR}, to prove $f$ is regular at $y$ it then suffices to show that $f' := f \circ v$ is regular at $y'$. Hence, at the expense of replacing $(X,Y,\Psi,f,g)$ with $(U', V', \Psi',f\circ v, g \circ u)$ we may assume without loss of generality that $g$ is regular everywhere and $\Psi$ is finite, flat and surjective. 

Choose an affine neighborhood $W$ of the set $g(\Psi^{-1}(y))$. Note that such an affine neighborhood exists since $Z$ is projective. Note that if $\{V_i\}_{i \in I}$ is an affine neighborhood base at $y$, then $\{\Psi^{-1}(V_i)\}_{i \in I}$ is an affine neighborhood base at $\Psi^{-1}(y)$ since $\Psi$ is closed and affine because it is finite. $g^{-1}(W)$ is a neighborhood of the set $\Psi^{-1}(y)$ by definition. So, we can pick an affine neighborhood $V$ of $y$ such that $\Psi^{-1}(V)$ is contained in $g^{-1}(W)$. This means that $g(\Psi^{-1}(V)) \subset W$. 

Let $U:= \Psi^{-1}(V)$. For every $x \in U \setminus \Psi^{-1}(\indet(f))$ we have 
\begin{equation}
\label{eqn:f-and-g}
f(\Psi(x)) = g(x) \in W.    
\end{equation}
 By the surjectivity of $\Psi$, we have $\Psi(\Psi^{-1}(\indet(f))) = \indet(f)$. This along with \eqref{eqn:f-and-g} implies that $f:V\setminus \indet(f) \lra Z$ factors through the inclusion $W \inj Z$. In other words, $f$ induces a morphism $f: V\setminus \indet(f) \lra W$. Note that $\indet(f)$ has codimension at least 2 because $Z$ is projective. So, because $V$ and $W$ are affine, Hartogs's extension theorem  implies that $f$ is regular on $V$ and we are done. 
\end{proof}

In our proof of Theorem \ref{thm:more-precise} in the case of toric surfaces, we know by Proposition \ref{prop:totally-inv} that there exists a subvariety $V$ of pure codimension 1 such that $\Phi^{-k}(V) = V$ and $\Phi$ induces an endomorphism on $X \setminus V$. The next lemma allows us to add to $V$ any other irreducible subvarieties of codimension 1 that happen to also be totally invariant under some iterate of $\Phi^k$ to get a new subvariety $\widetilde{V}$. Moreover, the map $\Phi$ will induce an endomorphism on $X\setminus V$. More precisely, we have
\begin{lemma}
\label{lem:add-more-totally-invs}
Suppose that $X$ is an irreducible projective variety and $\Phi: X \dra X$ is a rational self-map such that $\Phi^k$ is a regular on $X$. Let $V$ be a proper subvariety such that $\Phi^{-k}(V) = V$ and $\Phi$ induces a quasi-finite surjective endomorphism on $X \setminus V$. Let $H$ be an irreducible hypersurface in $X$ not contained in $V$ having the property that $\Phi^{-nk}(H) = H$ for some $n \ge 1$. Then, there exists a proper subvariety $\widetilde{V}$ containing both $V$ and $H$ such that $\Phi^{-k}(\widetilde{V}) = \widetilde{V}$ and $\Phi$ induces a quasi-finite surjective endomorphism on $X \setminus \widetilde{V}$. 
\end{lemma}
\begin{proof}
Let $H_1 := H \cap(X \setminus V)$ and define $\OO_\Phi(H)$ to be the union of the algebraic closure of all hypersurfaces in the orbit of $H_1$ under $\Phi$. In other words, 
\[
\OO_\Phi(H) := \bigcup_{i=0}^{nk - 1} \overline{\Phi^i(H_1)}
\]
Define 
\[
\widetilde{V} := V \cup \OO_\Phi(H).
\]
We claim that $\Phi: X \setminus V \lra X\setminus V$ induces an endomorphism on $X \setminus \widetilde{V}$. Since $\Phi|_{X \setminus \widetilde{V}}$ is regular it suffices to show that $\Phi(X \setminus \widetilde{V}) \subseteq X \setminus \widetilde{V}$. Suppose for the sake of contradiction that $\Phi(x) \in \widetilde{V}$ for some $x\in X \setminus \widetilde{V}$. This means that $\Phi(x) \in \OO_\Phi(H)$ i.e. $\Phi(x) \in \overline{\Phi^i(H_1)}$ for some $1 \le i \le nk$. Because $x \in X \setminus V$ and $\Phi$ is an endomorphism on $X \setminus V$ we must actually have $\Phi(x) \in \overline{\Phi^{i}(H_1)} \setminus V$. Since $\Phi$ is surjective on $X \setminus V$ there must exist $y \in X \setminus V$ such that $x = \Phi^{i-1}(y)$. So, we must have $\Phi^{i}(y) \in \overline{\Phi^i(H_1)} \setminus V$. Composing both sides with $\Phi^{nk - i}$ we get 
\[
\Phi^{nk}(y) \in \Phi^{nk - i}(\overline{\Phi^i(H_1)} \setminus V) \subseteq \Phi^{nk}(H) = H,
\]
where the last inclusion follows from the fact that the Zariski closure of $\Phi^{nk - i}(\overline{\Phi^i(H_1)} \setminus V)$ in $X$ must be irreducible and it contains $\Phi^{nk}(H_1)$. Hence, the Zariski closure is equal to $\Phi^{nk}(H)$. Therefore, we have 
\[
y \in \Phi^{-nk}(H) = H.
\]
But, $y$ lies inside $X \setminus V$ which implies that $y \in H_1$. Hence, $x = \Phi^{i-1}(y) \in \Phi^{i-1}(H_1) \subseteq \widetilde{V}$. This contradicts our assumption that $x \in  X \setminus \widetilde{V}$. 

To finish the proof, note that since $\Phi$ induces an endomorphism of $X \setminus \widetilde{V}$ we must have $\Phi^{-k}(\widetilde{V}) \subseteq \widetilde{V}$. Then, it follows from surjectivity of $\Phi^k$ that $\Phi^{-k}(\widetilde{V}) = \widetilde{V}$. Surjectivity and quasi-finiteness of the induced map on $X \setminus \widetilde{V}$ follows easily from  surjectivity and quasi-finiteness of $\Phi$ and the definition of $\widetilde{V}$. 
\end{proof}
The next step in proving Proposition \ref{prop:totally-inv} is to employ Lemma \ref{lem:regularComposition} to prove that if $\Phi^k$ is regular, then the indeterminacy loci of $\Phi, \dots,\Phi^{k-1}$ are totally invariant under $\Phi^k$. This shows the first part of Proposition \ref{prop:totally-inv} since the indeterminacy loci must all have codimension 2. The argument is almost exactly the same as the proof of \cite[Proposition 1.2]{BGR}. We include the details here for the sake of completeness. 
\begin{proposition}
\label{prop:exceptional}
Suppose that $\Phi: X \dra X$ is a dominant rational self-map such that $\Phi^k$ is regular for some $k \ge 2$ and $k$ is the smallest such $k$. Let $W_m$ be the indeterminacy locus of $\Phi^m$ for every $m \ge 1$. Then for every $1 \le r \le k-1$ and every $n \ge 0$ we must have $\Phi^{-k}(W_r) = W_r$ and $W_r = W_{nk + r}$.
\end{proposition}

\begin{proof}
We only prove this for $r = 1$ and the proof for $r = 2, \dots, r$ is identical. We have a descending chain of closed subsets 
\[
W_1 \supset W_{k + 1} \supset W_{2k + 1} \supset \cdots
\]
that must terminate. So, for some $\ell \ge 1$ we have $W := W_{k\ell + 1} = W_{k(\ell + 1) + 1} = \cdots$ . Then, by Lemma \ref{lem:regularComposition} for any $n \ge 1$ we have 
\[
\Phi^{-k}(W_{k\ell + 1}) = W_{k(\ell + 1) + 1}.
\]
we conclude that $\Phi^{-k}(W) = W$. 

To finish the proof we need to show that $W := W_{kn + 1}$ for every $n \ge 1$. Suppose for the sake of contradiction that there is $y \in W_{kn + 1} \setminus W$ for some $n < \ell$. Since $\Phi^k$ is surjective there exists $x_1 \in X$ such that $\Phi^{k(\ell - n)}(x_1) = y$. Also, $x_1 \notin W$ as $W$ is totally invariant under $\Phi^k$ and $y \notin W$. We have
\[
\Phi^{kn + 1}(\Phi^{k(\ell - n)}(x)) = \Phi^{k\ell + 1}(x),
\]
and both $\Phi^{k\ell + 1}$ and $\Phi^{k(\ell - n)}$ are regular at $x_1$. So, by Lemma \ref{lem:regularComposition} we conclude that $\Phi^{kn + 1}$ is regular at $y$ which is a contradiction.
\end{proof}
\begin{proposition}
Suppose that $\Phi^{-k}(W) = W$ for some subvariety $W$. Then, $\Phi^{-(mk + r)}(W) = \Phi^{-r}(W)$ for every $r,k \ge 0$.
\end{proposition}
\begin{proof}
Suppose that $x \in \Phi^{-r}(W)$ then it is clear that $\Phi^{mk + r}(x) \in W$. Now, suppose that $\Phi^{mk + r}(x) \in W$. Since $\Phi^{mk+r}$ is regular at $x$ by the definition of $x$ and $W_r = W_{mk + r}$ we have that 
$\Phi^r$ is also regular at $x$. Hence, because $\Phi^{mk+r}(x) \in W$ we have $\Phi^r(x) \in \Phi^{-mk}(W) = W$ which implies that $x \in \Phi^{-r}(W)$.
\end{proof}

Proposition \ref{prop:exceptional} finishes the proof of the first part of Proposition \ref{prop:totally-inv}. The next proposition proves the second part.

\begin{proposition}
\label{prop:totally-inv-case-by-case}
Suppose that $X$ is a $\bP^1$-bundle over a smooth curve $C$ or $X$ is $\bP^n$. Then, there exists a subvariety $\widetilde{V}$ of pure codimension 1 such that $\Phi^{-k}(\widetilde{V})=\widetilde{V}$ and $\Phi$ induces a quasi-finite and surjective endomorphism on $X \setminus \widetilde{V}$. 
\end{proposition}

Before proving Proposition \ref{prop:totally-inv-case-by-case} we define
 $V_0 := \cup_{i = 1}^{k-1} W_i$. Then, $\Phi^{-k}(V_0) = V_0$ by Proposition \ref{prop:exceptional}. Next, we define 
\begin{equation}
\label{eqn:def-V}
    V_i := \Phi^{-i}(V_0) \text{ for } i = 1,\dots,k-1\text{ and } V := \cup_{i=0}^{k-1} V_i.
\end{equation} 
Our goal is to show that $\Phi$ induces a quasi-finite and surjective endomorphism on the complement of $V$ (See Corollary \ref{cor:induced-on-comp}). This is a corollary of Lemmas \ref{lem:V-closed} through \ref{prop:W-inv} which we will prove next. 
\begin{lemma}
\label{lem:V-closed}
$V$ is Zariski closed in $X$.
\end{lemma}
\begin{proof}
Note that $V_0$ is closed as it is a finite union of closed subsets of $X$. Each $V_i$ is closed in $X\setminus W_i$ being the preimage of a closed set. So, $V_i \cap (X \setminus V_0)$ must be closed in $X \setminus V_0$. It follows that $\left(\cup_{i=1}^{k-1} V_i\right) \cap (X \setminus V_0)$ is closed in $X \setminus V_0$. Therefore, $V = V_0 \cup \left(\cup_{i=1}^{k-1} V_i\right)$ is closed in $X$. 
\end{proof}
\begin{lemma}
\label{prop:phi-reg}
$\Phi: X \setminus V \dra X \setminus V$ is regular. 
\end{lemma}
\begin{proof}
By the definition $\Phi: X \setminus V \dra X$ is regular. So, we need to show that if $\Phi(x) \in V$ then $x \in V$. It suffices to show this only for $\Phi(x) \in V_{k-1}$ by the definition of $V$. Note that this implies that $\Phi^k(x) \in V_0$ which means that $x \in V_0$ as $V_0$ is totally invariant under $\Phi^k$.   
\end{proof}

\begin{lemma}
\label{prop:W-inv}
$V$ is totally invariant under $\Phi^k$.
\end{lemma}
\begin{proof}
The inclusion $\Phi^{-k}(V) \subset V$ is clear by Lemma \ref{prop:phi-reg}. If this inclusion is strict we would get a descending chain
\[
V \supset \Phi^{-k}(V) \supset \Phi^{-2k}(V) \supset \cdots
\]
of strict inclusions by surjectivity of $\Phi^{k}$. This yields a contradiction. So, $\Phi^{-k}(V) = V$. 
\end{proof}

\begin{corollary}
\label{cor:induced-on-comp}
If we assume that $\Phi^k$ is a quasi-finite and surjective endomorphism, then the induced map $\Phi: X \setminus V \lra X \setminus V$ is also quasi-finite and surjective. 
\end{corollary}
\begin{proof}
$\Phi$ is quasi-finite since $\Phi^k$ is. To prove surjectivity note that $\Phi^k$ is surjective and for any $y \in X \setminus V$ there must exist $x \in X$ such that $\Phi^k(x) = y$. But, since $V$ is totally invariant under $\Phi^k$ this implies that $x \in X \setminus V$. Hence, $\Phi^{k-1}(x) \in X \setminus V$ must map to $y$ under $\Phi$. 
\end{proof}

\begin{proof}[Proof of Proposition \ref{prop:totally-inv-case-by-case}]
Let $V$ be as defined in equation \eqref{eqn:def-V}. We start with the case of a $\bP^1$-bundle over a smooth curve $C$. If $V$ has pure codimension 1 then we are done by Lemma \ref{cor:induced-on-comp}. If not, let $p_1, \dots, p_\ell$ be the codimension 2  components (i.e. isolated points) of $V$. Let $\pi: X \lra C$ be the projection. Then, $\Phi^{2k}$ preserves the fibers of $\pi$ by Lemma \ref{lem:MSS-5.8}. So, there must exist a morphism $g: C \lra C$ such that $g \circ \pi = \pi \circ \Phi^{2k}$. This implies that $\pi(p_i)$ must be totally invariant under $g$. Therefore, $F_i := \pi^{-1}(\pi(p_i))$ is totally invariant under $\Phi^{2k}$. Note that all of the fibers $F_i$ are not contained in $V$ since they contain $p_i$ and $p_i$ was assumed to be a codimension $2$ component of $V$. By Lemma \ref{lem:add-more-totally-invs} there exists a subvariety $\widetilde{V}$ containing $V$,$F_1, \dots, F_\ell$ such that $\Phi^{-k}(\widetilde{V}) = \widetilde{V}$ and the restriction of $\Phi$ to $X\setminus \widetilde{V}$ is a surjective and quasi-finite morphism. Note that by the construction given in the proof of Lemma \ref{lem:add-more-totally-invs}, the components of $\widetilde{V}$ of codimension 2 can only be the components of $V$ of codimension 2 i.e. $p_1,\dots,p_\ell$. But we know that the points $p_1,\dots,p_\ell$ are contained in $F_1,\dots,F_\ell$ which proves that $\widetilde{V}$ does not have any components of codimension 2. Hence, it must have pure codimension 1.  

In the case of $X = \bP^n$, we will prove that the subvariety $V$ defined in \eqref{eqn:def-V} must have codimension 1. Assume for the sake of contradiction that $V$ has codimension at least 2. The self-map $\Phi$ is given by $$[x_0: x_1 : \cdots : x_n] \mapsto [F_0: F_1: \dots: F_n]$$
where $F_0,F_1,\dots, F_n$ are homogeneous polynomials of the same degree. We claim that $\Phi^i$ must be given in homogeneous coordinates by 
$$[x_0: x_1 : x_2] \mapsto [F_0 \circ \Phi^{i-1}: F_1 \circ \Phi^{i-1}: \dots : F_n \circ \Phi^{i-1}] $$
for every $i \ge 1$, meaning that the homogeneous polynomials $F_0 \circ \Phi^{i-1}, F_1 \circ \Phi^{i-1}, \dots, F_n \circ \Phi^{i-1}$ do not have a common factor. In fact, if a homogeneous polynomial $G$ divides all $F_0 \circ \Phi^{i-1}, F_1 \circ \Phi^{i-1},\dots, F_n \circ \Phi^{i-1}$ then $Z(G)$ must be mapped to $\indet(\Phi) \subseteq V$ by $\Phi^{i-1}$. Corollary \ref{cor:induced-on-comp} then yields a contradiction, as $\indet(\Phi)$ has codimension at least 2, but the image of $Z(G)\setminus V$ must have codimension 1 since $\Phi^{i-1}$ restricted to $\bP^n \setminus V$ is quasi-finite. Thus, $\Phi^k$ is given by 
$$[x_0: x_1 : x_2] \mapsto [F_0 \circ \Phi^{k-1}: F_1 \circ \Phi^{k-1}: \dots : F_n \circ \Phi^{k-1}] $$
 and therefore must be undefined on $\text{Indet}(\Phi)$ which is a contradiction. 

 Now, if $V$ has pure codimension 1 we are done by Corollary \ref{cor:induced-on-comp}. So, suppose not and write $V = V' \cup V''$ where $V'$ is the union of the irreducible components of $V$ with codimension 1 and $V''$ is the union of the rest of the components. Let $U := \bP^n \setminus V$ and $U' = \bP^n \setminus V'$. $\Phi$ induces a morphism $\Phi: U \lra U'$. Note that $U'$ is an affine subscheme of $\bP^n$ as $V'$ has pure codimension 1 and $U$ is an open subscheme of $U'$ such that $\text{codim}(U' \setminus U) \ge 2$. This means that the morphism $\Phi: U \lra U'$ extends uniquely to an endomorphism on $U'$. Hence, if we define $\widetilde{V} := V'$ we will get an induced morphism on $X \setminus \widetilde{V}$ by $\Phi$. Similar to Lemma \ref{prop:W-inv} we conclude that $\Phi^{-k}(\widetilde{V}) = \widetilde{V}$. Lastly, the surjectivity and quasi-finiteness of the induced map follow from surjectivity and quasi-finiteness of $\Phi^k$. 
\end{proof}
Proposition \ref{prop:totally-inv-case-by-case} finishes the proof of the second part of Proposition \ref{prop:totally-inv}.

\section{Existence of normal models for eventually regular self-maps}
\label{sec:good-normal-model}

In this short section we collect two results. The first one uses the same argument as \cite[Proposition 2.12]{favre} to prove the existence of normal models for eventually regular self-maps $\Phi: X \dra X$. The second result uses \cite{Nak-normal-end} and \cite{Nak-Moishezon} to describe the structure of such normal models when $X$ is a projective surface and $\Phi$ is not birational.  
\begin{proposition}
\label{prop:good-normal-model}
Let $X$ be a variety over an algebraically closed field $K$. Suppose $\Phi: X \dra X$ is a dominant rational self-map of $X$ over $K$ such that $\Phi^k$ is regular for some $k \ge 1$. Then, there exists a normal birational model $\pi: \hat{X} \lra X$ and an endomorphism $F: \hat{X} \lra \hat{X}$ such that the next diagram commutes
\[
\begin{tikzcd}
\hat{X} \arrow[r, "F"] \arrow[d, "\pi"] & \hat{X} \arrow[d, "\pi"] \\
X \arrow[r, dashed, "\Phi"] & X.
\end{tikzcd}
\]
\end{proposition}
\begin{proof}
Let $U$ be the intersection of the domains of $\Phi, \dots, \Phi^{k-1}$. Then, $U$ is a dense Zariski open subset of $X$. Define
\[
\Gamma := \{(x,\Phi(x),\dots, \Phi^{k-1}(x)): x \in U\} \subset X^{k}
\]
Let $F$ be the endomorphism of $X^k$ defined by 
\[
(x_1,\dots,x_k) \mapsto (x_2, \dots,x_k, \Phi^k(x_1)). 
\]
If we let $\overline\Gamma$ be the Zariski closure of $\Gamma$, then we claim that $F(\overline\Gamma) = \overline{\Gamma}$. Indeed, for all $x \in U \cap \Phi^{-1}(U)$, the image of $(x,\Phi(x),\dots,\Phi^{k-1}(x))$ is equal to $(\Phi(x), \dots, \Phi^k(x))$ which is clearly in $\Gamma$. Since $U \cap \Phi^{-1}(U)$ is dense in $U$ we see that a dense subset of $\Gamma$ is sent to $\Gamma$. It then follows easily from irreducibility of $\Gamma$ that $F(\overline{\Gamma}) = \overline{\Gamma}$. Therefore, if we let $\pi_1: X^k \lra X$ be the projection onto the first coordinate we get the following commutative diagram
\[
\begin{tikzcd}
\overline{\Gamma} \arrow[r, "F"] \arrow[d, "\pi_1"] & \overline{\Gamma} \arrow[d, "\pi_1"] \\
X \arrow[r, dashed, "\Phi"] & X.
\end{tikzcd}
\]
We note that $\pi_1: \overline{\Gamma} \lra X$ is birational by the definition of $\Gamma$. Let $\pi: \hat{X} \lra \overline{\Gamma}$ be the normalization of $\overline{\Gamma}$. Then, by the univeral property of normalization, $F \circ \pi$ must factor as $\pi \circ \widetilde{F}$ for some endomorphism $\widetilde{F}: \hat{X} \lra {\hat{X}}$ and we get the next commutative diagram
\[
\begin{tikzcd}
\hat{X} \arrow[d, "\pi"] \arrow[r, "\widetilde{F}"] & \hat{X} \arrow[d, "\pi"] \\
\overline{\Gamma} \arrow[r, "F"] \arrow[d, "\pi_1"] & \overline{\Gamma} \arrow[d, "\pi_1"] \\
X \arrow[r, dashed, "\Phi"] & X.
\end{tikzcd}
\]
This finishes the proof.
\end{proof}
Proposition \ref{prop:good-normal-model} naturally motivates the next question.
\begin{question}
In Proposition \ref{prop:good-normal-model}, assuming $X$ is smooth, can we choose the normal model $\hat{X}$ to also be smooth?
\end{question}

When $\Phi^k$ is non-invertible, the induced endomorphism $F$ on $\hat{X}$ in Proposition \ref{prop:good-normal-model} is also non-invertible. With the additional assumption that $X$ is a projective surface, this imposes restrictions on what $\hat{X}$ could be. In fact, \cite[Theorem 1.1]{Nak-normal-end} gives a complete classification of normal projective surfaces admitting non-invertible endomorphisms. When $X$ is smooth and $\kappa(X) = -\infty$, Proposition \ref{prop:totally-inv} along with \cite{Nak-Moishezon} allows us to do even better. Putting these together we have the next result. 

\begin{proposition}
\label{prop:structure-of-the-model}
Suppose we are in the situation of Proposition \ref{prop:good-normal-model} and $X$ is a projective surface and $\Phi^k$ is non-invertible. Then, either there is a finite Galois cover $\nu: V \lra \hat{X}$ such that one of the conditions I-1 through I-6 in \cite[Theorem 1.1]{Nak-normal-end} hold, or $\hat{X}$ is rational with only quotient singularities. If we assume further that $X$ is smooth with $\kappa(X) = -\infty$, then there is an endomorphism $F_V: V \lra V$ such that $\nu \circ F_V = F^\ell \circ \nu$ for some $\ell \ge 1$. 
\end{proposition}
\begin{proof}
The first assertion is clearly a consequence of \cite[Theorem 1.1]{Nak-normal-end}. For the second assertion note that by Proposition \ref{prop:totally-inv} there exists a subvariety $H$ of $X$ of pure codimension 1 such that $\Phi^{-k}(H) = H$. Using Proposition \ref{prop:good-normal-model}, we have the next commutative diagram
\[
\begin{tikzcd}
\hat{X} \arrow[r, "F^k"] \arrow[d, "\pi_1 \circ \pi"] & \hat{X} \arrow[d, "\pi_1 \circ \pi"] \\ 
X \arrow[r, "\Phi^k"] & X,
\end{tikzcd}
\]
where $\pi_1$ and $\pi$ are as defined in the proof of Proposition \ref{prop:good-normal-model}. Since $\Phi^{-k}(H) = H$ we see that $$F^{-k}((\pi_1 \circ \pi)^{-1}(H)) = (\pi_1 \circ \pi)^{-1}(\Phi^{-k}(H)) = (\pi_1 \circ \pi)^{-1}(H).$$ Let $H_1$ be an irreducible component of $(\pi_1 \circ \pi)^{-1}(H)$ of codimension 1 in $\hat{X}$. It is then easy to see that $F^{-kt}(H_1) = H_1$ for some $t \ge 1$. The second assertion then follows immediately from \cite[Theorem A]{Nak-Moishezon}.
\end{proof}

\section{The case of toric surfaces}
\label{sec:toric}
The main results of this section are theorems \ref{thm:P2} and \ref{thm:toric}. As a corollary of these two we get the first conclusion of Theorem \ref{thm:main} and Thoerem \ref{thm:more-precise} in the case of toric surfaces. 
\subsection{The case of \texorpdfstring{$\bP^2$}{P2}.} The next theorem is a more precise version of the first part of Theorem \ref{thm:main}. 
\begin{theorem}
\label{thm:P2}
Suppose $\Phi: \bP^2 \dra \bP^2$ is a dominant rational self-map that is not regular. Suppose $\Phi^k$ is regular and non-trivial for some $k \ge 2$. Then, there must exist a subvariety $V$ which is a union of lines in $\bP^2$ such that $\Phi^{-k}(V) = V$ and $\Phi$ induces an endomorphism on $\bP^2 \setminus V$. Moreover, one of the following must hold:
\begin{itemize}
    \item[(a)] $V$ is a union of at most two lines and $\Phi^2$ is regular, or

    \item[(b)] $V$ is a union of exactly three lines and $\Phi^i$ is regular for some $i \in \iterates$. 
\end{itemize}
\end{theorem}
\begin{proof} 
We first note that by Proposition \ref{prop:exceptional} there must exist a subvariety $V$ of pure codimension 1 such that $\Phi^{-k}(V) = V$. Using Theorem \ref{conj:lin-conj} the degree of all the irreducible components of $\mathcal{V}$ is equal to 1. We note that it suffices to prove the theorem after changing the coordinates i.e. after replacing $\Phi$ with $L \circ \Phi \circ L^{-1}$ for some automorphism $L$ of $\bP^n$. Hence, we may use \cite[Theorem 4.5]{FS} to assume we are in one of the following cases:
\begin{enumerate}
    \item If $V$ has one irreducible component give by $\{Z = 0\}$ and $\Phi^k$ is of the form $[f(X,Y,Z):g(X,Y,Z):Z^d]$ where $f(X,Y,0)$ and $g(X,Y,0)$ have no common factors. 

    \item  If $V$ has two irreducible component given by $\{Z = 0\}$ and $\{Y = 0\}$ and $\Phi$ is of the form $[f(X,Y,Z):Y^d:Z^d]$ where $f(X,0,0) = X^d$.

    \item If $V$ has three irreducible components given by $\{Z = 0\}, \{Y = 0\},$ and $\{X = 0\}$ and $\Phi$ is of the form $[X^d:Y^d:Z^d]$. 
\end{enumerate}
Before handling these cases we let $\lambda_1$ and $\lambda_2$ be the dynamical and topological degrees of $\Phi$ as defined in Definition \ref{def:dynamical-degrees}. 

\textbf{Case 1.} After identifying $\C^2$ with $\bP^2 \setminus V$ and using Proposition \ref{prop:totally-inv} we see that the map $\Phi$ induces a polynomial endomorphism of $\C^2$. Since $\Phi^k$ is regular on $\bP^2$ it is clear that $\lambda_1^2 = \lambda_2$ and that $\deg(\Phi^n)/\lambda_1^n$ remains bounded. Hence, by Proposition \ref{prop:FJ-5.3} we must have $\mathcal{T}_{\Phi^{2n}} = \mathcal{T}_{\Phi^2}$ for every $n \ge 1$. In particular, $\mathcal{T}_{\Phi^{2k}} = \mathcal{T}_{\Phi^2}$ which means that the valuation $-\deg$ is an eigenvaluation of $\Phi^2_\bullet$. Further, we have by Proposition \ref{prop:FJ-5.3} that $-\deg$ is totally invariant under $\Phi^2_\bullet$ which means that $\Phi^2$ must extend to an endomorphism of $\bP^2$.

\textbf{Case 2.}
In this case, after replacing $\Phi$ with $\Phi^2$ we may assume without loss of generality that $\Phi$ has the following form on $\C^2 = \bP^2 \setminus \{Z = 0\}$

\begin{equation}
\label{eqn:Phi-Form}
(x,y) \mapsto (x^ea_e(y) + \cdots + a_0(y), cy^d),   
\end{equation}
for $a_0, \dots, a_e \in \C[y, y^{-1}]$  and some $c \in \C$ and some $d \ge 1$. It is clear that the topological degree of $\Phi$ is equal to $de$. But, since $\Phi^k$ is regular on $\bP^2$ and is of the form $(f(x,y), c^{d^{k-1} + \cdots + 1}y^{d^k})$ for some $f \in \C[x,y]$, the topological degree of $\Phi^k$ is equal to $d^{2k}$. This implies that $(de)^k = d^{2k}$ which means that $d = e$.

Now, we will show that $\Phi$ must be regular on $\C^2$. Indeed, $\Phi^n$ is given by 
\[
(x,y) \mapsto (x^{d^n}a_{d^n,n}(y) + \cdots + a_{0,n}(y), c^{d^{n-1} + \cdots + 1}y^{d^n}).
\]
Since $\Phi^k$ is regular on $\C^2$ it is clear that $a_{d^k,k}$ must be a constant and the rest of the coefficients are polynomials in $\C[y]$. We can then use case (c) of Lemma \ref{lem:skew-then-split} to show that the coefficients $a_0(y),\dots,a_{d-1}(y)$ must lie in $\C[y]$ and $a_d(y) \in \C^\ast$. 

So, we assume from now on that $\Phi$ is regular on $\C^2$ and $a_d(y)$ is a constant. Using equation \eqref{eqn:Phi-Form} it is clear that the for all but finitely many $\alpha \in \C$, $\Phi$ maps the curve $y = \alpha$ onto $y = \alpha^d$. This implies that the lift of $\Phi$ to $X := \text{Bl}_{[1:0:0]}(\bP^2)$ fixes the exceptional divisor $E$. We recall that $\alpha(\nu_E) = 0$ and $A(\nu_E) = -1$ (See \cite[Section 9.3.3]{Jon11}) which means that $\nu_E \in \mathcal{V}_1$ and hence $\nu_E \in \T_{\Phi}$ (See also \cite[Lemma 1.4]{FJ07}). This means that $\T_\Phi$ is an interval and $\T_{\Phi} = \T_{\Phi^n}$ for $n \ge 1$ by Proposition \ref{prop:FJ-5.3}. We conclude that $-\deg \in \T_\Phi$ which implies that $\Phi$ extends to a regular endomorphism of $\bP^2$.

\textbf{Case 3.} Again using Proposition \ref{prop:totally-inv}, the map $\Phi$ must be an endomorphism of $X\setminus V  = \bP^2 \setminus \{XYZ = 0\} \simeq\bG_m^2$. So, $\Phi$ is given by a translation composed with a group endomorphism of $\bG_m^2$ corresponding to a matrix $A \in M_{2,2}(\Z)$ (See \cite[Theorem 2]{Iitaka}). Moreover, since $\Phi^k$ is of the form $[X^d:Y^d:Z^d]$ we must have $A^k = d\cdot \text{Id}$. So, by Lemma \ref{lem:smallestPower} we conclude that $A^i$ is equal to $d' \cdot \Id$ for some $d' \ge 1$ and some $i \in \iterates$. Thus, $\Phi^i$ extends to $\bP^2$ for some $i \in \iterates$ and we are done.  
\end{proof}
\begin{remark}
\label{rem:long-then-Gm-P2}
As a consequence of the proof of Theorem \ref{thm:P2}, we see that if $\Phi$ is eventually regular and takes more than 2 iterations to become regular, then $\Phi$ must induce an endomorphism on an open subvariety $U$ of $X$ that is isomorphic to $\bG_m^2$.  
\end{remark}
\subsection{The case of Hirzebruch surfaces} 
\label{subsec:Hirzebruch}
We now focus on the proof of Theorem \ref{thm:more-precise} when $X$ is a Hirzebruch surface. Recall that to each $n \ge 0$ we can associate the Hirzebruch surface $\F_n := \bP(\OO_{\bP^1} \oplus \OO_{\bP^1}(n))$. Hirzebruch surfaces have the structure of a $\bP^1$-bundle by their definition. We will make use of this structure in our proof of Theorem \ref{thm:more-precise}. 

We start by handling the case of $X = \bP^1 \times \bP^1$.
\begin{proposition}
\label{prop:P1-times-P1}
Suppose that we are in the situation of Theorem \ref{thm:more-precise} and $X = \bP^1 \times \bP^1$. Then, $\Phi^i$ is regular for some $i \in \{1,2,3,4,6\}$.
\end{proposition}
\begin{proof}
Let $\pi_1$ and $\pi_2$ be the two projections of $\bP^1 \times \bP^1$. If $\Phi$ is regular we are done so we assume not. Then, by Proposition \ref{prop:totally-inv} there must exist a subset $W$ of codimension 2 i.e. a set of finitely many points such that $\Phi^{-k}(W) = W$. Take a point $p \in W$. There must exist some $\ell \ge 1$ such that $\Phi^{-k\ell}(p) = p$. After a suitable change of coordinates we assume that $p = (\infty, \infty)$. By Lemma \ref{lem:MSS-5.8} we know that $\Phi^{2k\ell}$ preserves the fibers of both $\pi_1$ and $\pi_2$. So, $\Phi^{2k\ell}$ must be given by 
\begin{equation}
\label{eqn:form-of-phi}
(x,y) \mapsto (f(x),g(y)),    
\end{equation}
for two polynomials $f$ and $g$. It follows that $F := \pi_1^{-1}(\infty)$ and $F' := \pi_2^{-1}(\infty)$ are totally invariant under $\Phi^{2k\ell}$. Using Lemma \ref{lem:add-more-totally-invs} we conclude that there exists a subvariety $V$ of $X$ containing $F$ and $F'$ such that $\Phi^{-k}(V) = V$ and $\Phi$ induces a surjective and quasi-finite endomorphism of $X \setminus V$. 

The only totally invariant subvarieties under a product map are the fibers of $\pi_1$ and $\pi_2$. Thus, $V$ can only consist of a union of these fibers. Suppose that $V$ contains the fibers $F_{a_1}, \dots, F_{a_n}$ of $\pi_1$ other than $F$ and the fibers $F'_{b_1},\dots, F'_{b_m}$ of $\pi_2$ other than $F'$ for $a_1,\dots,a_n,b_1,\dots,b_m \in \A^1$. Then, $\Phi$ must induce an endomorphism of $\A^1 \setminus \{a_1 ,\dots,a_n\} \times \A^1 \setminus \{b_1,\dots,b_m\}$. Suppose without loss of generality that $n \le m$. If $m \ge 2$, then at least three fibers of $\pi'$, namely $F'_\infty, F'_{b_1},\dots,F'_{b_m}$, are totally invariant under some iterate of $\Phi^k$. So, $g$ must be an automorphism of finite order. If $n$ is also at least 2 we conclude that $f$ is also an automorphism which contradicts the assumption that $\Phi^k$ is non-invertible. So, assume $n < 2 \le m $. 

The only morphisms from $\A^1 \setminus \{a_1 ,\dots,a_n\}$ to $\A^1 \setminus \{b_1,\dots,b_m\}$ are the constant morphisms. This implies that $\pi' \circ \Phi$ must factor through $\pi'$. In other words, $\Phi$ is given by a skew product of the form
\[
(x,y) \mapsto (s(x,y), t(y)),
\]
for some $s \in \C(x,y)$ and some $t \in \C(y)$ of finite order. Note that $n$ can be either $0$ or $1$. After a suitable change of coordinates we may assume that $a_1 = 0$ if $n = 1$. In any case, at the expense of replacing $\Phi$ with $\Phi^2$, we can assume that $s(x,y)$ must be given by a polynomial with coefficients in $\C(y)$. Since an iterate of $\Phi$ is a product and $t$ has finite order we conclude by case (1) of Lemma \ref{lem:skew-then-split} that $\Phi^2$ itself is a product. Therefore, it must be regular everywhere on $\bP^1 \times \bP^1$. So, we may assume $m,n \le 1$, that is, $V$ contains at most two vertical or horizontal fibers.   

\textbf{Case 1.} $V = F \cup F'$. In this case $\Phi$ induces an endomorphism on $X \setminus (F \cup F') \simeq \A^2$. Let $d$ and $e$ be the degrees of the polynomials $f$ and $g$ defined in equation \eqref{eqn:form-of-phi}. If $d = e$, then $\lambda_1(\Phi)^2 = \lambda_2(\Phi) = d^2$. Note that $\Phi^{2k\ell}$ fixes the rational pencil valuations corrsponding to $x = const$ and $y  = const$. It follows by Proposition \ref{prop:FJ-5.3} that $\Phi^2$ must fix these valuations. Therefore, $\Phi^2$ must also be given by a product of the form $(x,y)\mapsto (f_1(x),g_1(y))$. So, $\Phi^2$ extends to $\bP^1 \times \bP^1$. If $d > e$, then $d^2 = \lambda_1(\Phi)^2 > \lambda_2(\Phi) = de$. So, the eigenvaluation of $\Phi$ is unique. Since $\Phi^{2k\ell}$ fixes the rational pencil valuations corrsponding to $x = const$ we conclude that $\Phi$ also fixes this valuation. So, $\Phi$ must be given by a skew product 
\[
(x,y) \mapsto (f_1(x), g_1(x,y)). 
\]
where $f_1 \in \C[x]$ and $g_1\in\C[x,y]$. We conclude by Lemma \ref{lem:skew-then-split} that $\Phi$ is a product and must extend to $\bP^1 \times \bP^1$. 

\textbf{Case 2.} $V$ contains a fiber other than $F$ and $F'$. Let $V = F \cup F' \cup F'_a$ where $a \in \C$ and $F_a = \pi_2^{-1}(a)$. We assume after a change of coordinates that $a = 0$. So, $\Phi$ must induce an endomorphism of $X \setminus (F \cup F' \cup F_0) \simeq \A^2 \setminus \{y = 0\}$. It follows that $\Phi^2$ is a skew product of the form 
\[
(x,y) \mapsto (p(x,y), cy^d),
\]
for some $d \ge 1$ and some $p \in \C[x,y, y^{-1}]$ and $c \in \C$. Similar to before, we assume using a suitable change of coordinates that $c = 1$. We conclude by Lemma \ref{lem:skew-then-split} that $\Phi^2$ must extend to $\bP^1 \times \bP^1$. 

\textbf{Case 3.} $V$ contains a horizontal and a vertical fiber. In this case $X \setminus V \simeq \bG_m^2$. Suppose $\Phi$ is given by a translation composed with the group endomorphism corresponding to the matrix $A \in M_{2,2}(\Z)$. For $\Phi^k$ to extend to $X$, some power $A^k$ of $A$ must be a diagonal matrix. It follows from Lemma \ref{lem:smallestPower}  that $\Phi^i$ extends to $\bP^1 \times \bP^1$ for some $i \in \{1,2,3,4,6\}$.   
\end{proof}

Next we prove the case of $X = \F_n$ for $n \ge 1$. 
\begin{proposition}
\label{prop:Hirzebruch}
Suppose that we are in the situation of Theorem \ref{thm:main} and $X = \F_n = \bP(\OO_{\bP^1} \oplus \OO_{\bP^1}(n))$ for some $n > 0$. Then, $\Phi^i$ is regular for some $i \in \iterates$. 
\end{proposition}
\begin{proof}
Let $\pi: X \lra \bP^1$ be the projection and $C_0$ be the zero section of this projection. Recall that $C_0$ is the unique irreducible curve of negative self-intersection in $X$ (See \cite[Proposition IV.1]{Bea96}). Therefore, $C_0$ must be totally invariant under $\Phi^k$ (See also \cite[Lemma 8]{Nak02}). By Lemma \ref{lem:MSS-5.8}, $\Phi^k: X \lra X$ preserves the projection $\pi: X \lra \bP^1$. In other words, there exists an endomorphism $f: \bP^1 \lra \bP^1$ such that $f \circ \pi = \pi \circ \Phi^k$. Moreover, using Lemma \ref{lem:MSS-5.10} and the fact that $\Phi^k$ is not an automorphism we know that $f$ is not an automorphism. If $\Phi$ is not regular but eventually regular then by Proposition \ref{prop:totally-inv} there exists a subvariety $W$  of codimension 2 that is totally invariant under $\Phi^k$. Hence, $\pi(W)$ is totally invariant under $f$. After a change of coordinates we assume without loss of generality that $\infty \in \pi(W)$. This means that the fiber $F_\infty = \pi^{-1}(\infty)$ is totally invariant under some iterate of $\Phi^k$. Using Lemma \ref{lem:add-more-totally-invs} we conclude that there exists a subvariety $V$ of pure codimension 1 containing both $C_0$ and $F_\infty$ such that $\Phi^{-k}(V) = V$ and the restriction of $\Phi$ to $X \setminus V$ is a surjective and quasi-finite endomorphism. Moreover, every component of $V$ other than $C_0$ and $F_\infty$ must also be totally invariant under some iterate of $\Phi^k$.
\begin{lemma}
\label{lem:two-comps}
$V$ can contain at most one fiber other than $F_\infty$. Also, $V$ has at most two components other than $C_0$ and $F_\infty$ and when there are exactly two such components, one of them must be a fiber. Moreover, any such component that is not a fiber must intersect each fiber at exactly one point and intersects $C_0$ in at most one point outside of $C_0 \cap F_\infty$.
\end{lemma}
\begin{proof}
Note that if $V$ contains two fibers $F_a$ and $F_b$ of $\pi$ for some $a, b \in \bP^1$ other than $\infty$, then some iterate $\Phi^{k\ell}$ must fix all $F_a$, $F_b$, and $F_\infty$. Thus, $f^\ell$ is the identity. This is a contradiction since we know that $f$ cannot be an automorphism. Hence, $V$ contains at most one fiber other than $F_\infty$. 

Now suppose that $C$ is the subvariety of $V$ not containing $C_0$ or any fibers. Let $d \ge 2$ be the degree of $\Phi^k$ when restricted to a fiber $F$ of $\pi$. Suppose that $C$ intersects the fiber $F_x$ of $\pi$ in at least two points at $\alpha_x$ and $\beta_x$, where $\alpha_x \ne \beta_x$ for a generic $x \in \bP^1$. Let $\gamma_x$ denote the intersection point $C_0 \cap F_x$. Since $C$ and $C_0$ are totally invariant, the points $\alpha_x$,  $\beta_x$, and $\gamma_x$ must be sent to $\alpha_{f(x)}, \beta_{f(x)}$, and $\gamma_{f(x)}$ with local degree $d$. This yields a contradiction since $d \ge 2$ and the points $\alpha_x$,  $\beta_x$, and $\gamma_x$ are distinct for a generic $x \in \bP^1$. So, $C$ must be irreducible and intersects each fiber at exactly one point. From this discussion, we also conclude that if there is a component in $V$ other than $C,C_0,$ or $F_\infty$, then that component cannot intersect the generic fiber of $\pi$. Hence, it must itself be a fiber of $\pi$.  

Finally, suppose $C$ intersects $C_0$ at points $p$ and $q$ outside of $C_0 \cap F_\infty$. Then, $p$ and $q$ must be totally invariant under some iterate of $\Phi^k$. It follows that $\pi(p), \pi(q)$ must be totally invariant under some iterate of $f$. Since $\infty$ is also totally invariant under some iterate of $f$, we conclude that $f$ is an automorphism which is a contradiction. 
\end{proof}
To finish, we break the proof into three cases based on the number of irreducible components of $V$. 

\textbf{Case 1.} $V$ consists of exactly two components other than $C_0$ and $F_\infty$. We assume that all components of $V$ are fixed by $\Phi^k$ after replacing $\Phi^k$ with an iterate. Using Lemma \ref{lem:two-comps}, one of these components is a fiber $F_a$ for some $a \in \bP^1$ and the other one is a curve $C$ intersecting each fiber at exactly one point. The complement $X \setminus F_\infty$ is isomorphic to $\A^1 \times \bP^1$ and we can view $C_0$ as corresponding to $\A^1 \times \{\infty\}$ under this isomorphism. Also, we can think of $C_0$ as being given by the equation $y = g(x)$ where $g$ is an endomorphism of $\bP^1$ and $x$ and $y$ are the coordinates of $\A^1$ and $\bP^1$, respectively. Note that $g$ (as a function on $\A^1$) can only have poles at $x = a$. Indeed, if $b \ne a$ is a pole of $g$, then since both $C$ and $C_0$ are totally invariant under some iterate of $\Phi^k$, $b$ must be totally invariant under some iterate of $f$. Hence, $a,b,$ and $\infty$ are all totally invariant under some iterate of $f$ and $f$ is an automorphism because of this. This results in a contradiction. We assume after a change of coordinates that $a = 0$. There exists an isomorphism of $i: X \setminus V \lra \C^2 \setminus \{xy = 0\} \simeq \bG_m^2$ given by 
\[
(x,y) \mapsto (x, y - g(x)).
\]
Using this isomorphism to change coordinates along with the fact that endomorphisms of $\bG_m^2$ are group endomorphisms composed with translations and that $\Phi^k$ preserves the fibers of the projection onto the first coordinate, we conclude that $\Phi^k$ must be given by 
\[
(x,y) \mapsto (t_0x^\alpha, t_1x^\beta(y - g(x))^\gamma + g(t_0x^\alpha)). 
\]
for some integers $\alpha,\beta,\gamma$ and some $t_0,t_1 \in \C^\ast$. We assume that $\alpha,\gamma \ge 1$ at the expense of replacing $\Phi^k$ by an iterate. Therefore, $\alpha = \gamma$ by Lemma \ref{lem:MSS-5.10}. Moreover, $\alpha,\gamma \ge 2$ since $\Phi^k$ is not an automorphism.
\begin{claim}
\label{claim:g-no-pole}
$\beta = 0$.    
\end{claim}
\begin{proof}
If $\ord_x(g) \ge 0$ then $g$ is a polynomial. Then, $\beta = 0$ since otherwise $\Phi^k$ contracts $\{x = 0\}$ to a point contradicting the finiteness of $\Phi^k$. Suppose $\ord_x(g) = \alpha < 0$ and write 
\[
g(x) = a_i x^i + a_j x^j + \cdots 
\]
for some $a_i, a_j \in \C^\ast$ where $i < \beta$ and $i < 0$. The order of vanishing at $x = 0$ of $t_1x^\beta(y - g(x))^\gamma$ and $g(t_0x^\gamma)$ are equal to $\gamma i + \beta$ and $\gamma i$, respectively. Thus, if $\beta \ne 0$ we must have 
\[
\ord_x(t_1x^\beta(y - g(x))^\gamma + g(t_0x^\gamma))) = \min\{\gamma i + \beta, \gamma i\} < 0,
\]
and thus $t_1x^\beta(y - g(x))^\gamma + g(t_0x^\gamma))$ has a pole at $x = 0$. Then $\Phi^k$ contracts the fiber $F_0$ the point $(0
, \infty)$ under $\Phi^k$ contradicting finiteness.
\end{proof}

To summarize, we have shown that some iterate of $i \circ \Phi \circ i^{-1}$ i.e. $i \circ \Phi^k \circ i^{-1}$ is given by 
\[
(u,v) \mapsto (t_0u^\gamma,t_1v^\gamma).
\]
Since $i \circ \Phi\circ i^{-1}$ also induces an endomorphism of $\bG_m^2$ it follows from Lemma \ref{lem:smallestPower}, as in Case 3 of the proof of Theorem \ref{thm:P2}, that $i \circ \Phi^j \circ i^{-1}$ must already be given by $d'\cdot Id$ for some $d' \ge 2$ and some $j \in \iterates$. Therefore, $\Phi^j$ is given by 
\begin{equation}
\label{eqn:form-of-phi^j}
(x,y) \mapsto (t_0'x^{d'}, t_1'(y - g(x))^{d'} + g(x^{d'})).
\end{equation}
for some $t_0',t_1' \in \C^\ast$. Note that $\Phi^k$ is regular on $X \setminus (C_0 \cup F_\infty) \cong \C^2$. So, an iterate of $\Phi^j$ must also be an endomorphism of $\C^2$. Using case (c) of Lemma \ref{lem:skew-then-split}, we conclude that $\Phi^j$ itself is an endomorphism of $\C^2$. Using equation \eqref{eqn:form-of-phi^j} we observe that $\Phi^j$ fixes the pencil valuation $x = const$ which is an end in $\V_1$, it follows from Proposition \ref{prop:FJ-5.3} that $\mathcal{T}_{\Phi^j} = \mathcal{T}_{\Phi^{j\ell}}$ for all $\ell \ge 0$. Since some iterate extends to an endomorphism of $X$, the divisorial valuation corresponding to $F_\infty$, which also belongs to $\V_1$, must be totally invariant under some iterate $\Phi^{j\ell}$. Hence, it is also totally invariant under $\Phi^j$. This means that $\Phi^j$ must extend to $X$. 

\textbf{Case 2.} $V$ consists of exactly one component other than $C_0$  and $F_\infty$. If this component is a fiber we may assume after a change of coordinates that it is equal to $F_0$. We conclude as in Case 2 of the proof of Theorem \ref{thm:P2} that $\Phi^2$ must induce an endomorphism on $X \setminus (C_0 \cup F_\infty) \simeq \C^2$ that fixes the rational pencil valuation $\{x = const\}$ which is an end in $\V_1$. We conclude using Proposition \ref{prop:FJ-5.3} that $\mathcal{T}_{\Phi^2} = \mathcal{T}_{\Phi^{2n}}$ for $n \ge 1$. Since some iterate of $\Phi^2$ extends to $X$, the divisorial valuations corresponding to $F_\infty$ and $C_0$ are totally invariant under some iterate of $\Phi^2_\bullet$. Because $\mathcal{T}_{\Phi^2} = \mathcal{T}_{\Phi^{2n}}$ for $n \ge 1$, they are also totally invariant under $\Phi^2$. Hence, $\Phi^2$ must extend to $X$. 

If the component is not a fiber then it is given by an equation of the form $y=g(x)$ for some endomorphism $g$ of $\bP^1$. Note that if $g$ has a pole at a point $a$ other than $\infty$ we conclude as in Case 1 that the fiber $F_a$ is totally invariant under some iterate of $\Phi^k$. Using Lemma \ref{lem:add-more-totally-invs} we can add this fiber to $V$ and conclude the proof by the proof of Case 1. So, we can assume that $g(x)$ is a polynomial in $x$. If we let $i$ be the map defined in the proof of Case 1, it follows that $\widetilde{\Phi} := i \circ \Phi \circ i^{-1}$ induces an endmorphism of $\C^2 \setminus \{y = 0\}$. We conclude by Case 2 of the proof of Theorem \ref{thm:P2} that $\widetilde{\Phi^2}$ must be an endomorphism of $\C^2$ fixing the rational pencil valuation $\{x = const\}$.  We conclude that $\mathcal{T}_{\widetilde{\Phi}^2} = \T_{\widetilde{\Phi}^{2n}}$ for $n \ge 1$. Since $i$ is an automorphism of $\C^2$ we see that $\Phi^2$ is also an endomorphism of $\C^2$. Note that $i$ is an automorphism of $\C^2$ and therefore it induces an isomorphism on $\V_1$. So, it follows that $\mathcal{T}_{\Phi^2} = \mathcal{T}_{\Phi^{2n}}$ for $n \ge 1$. The fact that $\Phi^2$ must extend to $X$ follows as before.

\textbf{Case 3.} $V$ has no component other than $C_0$ and $F_\infty$. In this case $\Phi$ induces an endomorphism of $X \setminus (C_0 \cup F_\infty) \simeq \C^2$. We conclude that $\Phi^2$ must extend to $X$ using Proposition \ref{prop:FJ-5.3} as before.

\end{proof}
\begin{remark}
\label{rem:long-then-Gm-P1-bundle}
Similar to the case of $X = \bP^2$, we see from the proof of Propositions \ref{prop:P1-times-P1} and \ref{prop:Hirzebruch} that if it takes more than two iterations for an eventually regular map $\Phi$ on a Hirzebruch surface to become regular, then $\Phi$ must induce an endomorphism on an open subvariety $U$ of $X$ that is isomorphic to $\bG_m^2$.\end{remark}

Propositions \ref{prop:P1-times-P1} and \ref{prop:Hirzebruch} give us the next corollary.
\begin{corollary}
\label{cor:P1-bundles-over-P1}
Suppose we are in the situation of Theorem \ref{thm:more-precise} and $X$ is a Hirzebruch surface. Then, $\Phi^i$ is regular for some $i \in \iterates$.
\end{corollary}

\subsection{The case of general toric surfaces}
\label{subsec:general-toric}
We are now ready to prove Theorem \ref{thm:more-precise} in the case of a general toric surface.
\begin{theorem}
\label{thm:toric}
Suppose that we are in the situation of Theorem \ref{thm:more-precise} and $X$ is a toric surface. Then, $\Phi^i$ is regular for some $i \in \iterates$.  
\end{theorem}
\begin{proof}
If $X$ is minimal then it is either isomorphic to $\bP^2$ or some $\F_n$ for $n=0$ or $n\ge2$. Then, we are done by Theorem \ref{thm:P2} and Corollary \ref{cor:P1-bundles-over-P1}. So, we assume that $X$ is not minimal. Let $S(X)$ be the set of all irreducible curves. By Lemma \ref{lem:S(X)-tot-inv} there must exist $m \ge 1$ such that $\Phi^{-mk}(C) = C$ for all $C \in S(X)$. Since $X$ is not minimal, there must exist a $(-1)$-curve $E$ in $X$. If we blow down this curve to get a surface $X_1$, then $\Phi^{mk}$ induces an endomorphism of $X_1$ since $\Phi^{-mk}(E) = E$. Continuing in this fashion, we get a finite sequence of blow-downs
\[
\begin{tikzcd}
X := X_0 \arrow[r, "b_0"] & X_1 \arrow[r,"b_1"] & \cdots \arrow[r] & X_{n-1} \arrow[r,"b_{n-1}"] & X_n
\end{tikzcd}
\]
such that $X_n$ is a minimal surface and some iterate of $\Phi^{k}$ induces an endomorphism of $X_i$ for $i = 0,\dots,n$. So, $X_n$ is either isomorphic to $\bP^2$ or a $\bP^1$-bundle over $\bP^1$ i.e. a Hirzebruch surface. Note that if $X_n = \bP^2$, then $X_{n-1}$ is the blow up of $\bP^2$ at one point which is isomorphic to the Hirzebruch surface $\F_1$ and thus is a $\bP^1$-bundle. If $X_n = \bP^1 \times \bP^1$, then $X_{n-1}$ is isomorphic to $\bP^2$ blown up at two points. So, we may assume that the minimal model of $X$ is $\bP^2$. So, again we can assume without loss of generality that $X_{n-1} = \F_1$. Hence, we assume, after replacing $X_n$ with $X_{n-1}$ if necessary, that $X_n$ is isomorphic to $\F_m$ for some $m \ge 1$. We let $f_i:X_i \dra X_i$ denote the rational map induced by $\Phi$ on $X_i$. 

 Since $f_n$ has an iterate that is regular on $X_n$ we conclude by Proposition \ref{prop:Hirzebruch} that $f_n^j$ is regular for some $j \in \iterates$. Let $\pi: X_n \lra \bP^1$ be the projection of $X_n$ and $C_0$ be the zero section of the projection. Then there must exist an endomorphism $g: \bP^1 \lra \bP^1$ such that $g\circ \pi = \pi \circ f_n^j$ by Lemma \ref{lem:MSS-5.8}. Our goal is to show that $f_n^{j'}$ must extend to $X$ for some $j' \in \iterates$ as this will show that $\Phi^{j'}$ is an endomorphism of $X$.  We break the proof of this claim into two cases.

 \textbf{Case 1.} There is only one fiber of $\pi$ that is totally invariant under some iterate of $f_n^j$: In this case we may assume that the fiber is totally invariant under $f_n^j$ itself. Also, without loss of generality we can take the fiber to be $F_\infty := \pi^{-1}(\infty)$. Let $\Psi := \Phi^j$. We will show that $\Psi$ itself extends to $X$. 

 Recall that $X$ was obtained from $X_n$ after finitely many point blow-ups. Suppose that $X_{n-1}$ is obtained by blowing up $X_{n}$ at some point $p \in X_n\setminus (C_0 \cup F_\infty)$. By our discussion in the beginning of the proof, the exceptional divisor of this blow up is totally invariant under an iterate of $f_{n-1}$. This implies that the point $p$ is totally invariant under some iterate of $f_{n}$. We conclude that the fiber $F_{\pi(p)} = \pi^{-1}({\pi(p)})$ is totally invariant under some iterate of $f_n$ which contradicts our hypothesis that there is a unique fiber that is totally invariant under some iterate. Thus, $p \in F_\infty$. We conclude inductively and using the same argument that $X$ must be obtained from $X_n$ from a series of blow ups centered at $C_0 \cup F_\infty$. 

 Since $C_0$ and $F_\infty$ are both totally invariant under $f_n$, $f_n$ induces a polynomial endomorphism on $\C^2 \simeq X \setminus (C_0 \cup F_\infty)$. Let $\Delta(X)$ denote the dual graph of $X$. We can view $\Delta(X)$ as a finite subtree of $\V_\infty$. To prove that $f_n$ extends to $X$ it then suffices to show that $\Delta(X)$ is totally invariant under the induced map $(f_n)_\bullet: \V_\infty \lra \V_\infty$. But, this is a consequence of Proposition \ref{prop:5.3-extension}. 

\textbf{Case 2.} There are at least two fibers of $\pi$ that are totally invariant under some iterate of $f_n^j$: Note that if there are three fibers that are totally invariant under some iterate of $f_n^j$, then $g$ must be an automorphism of finite order. This contradicts the fact that $f_n^j$ is not an automorphism. Thus, we assume that there are exactly two fibers that are totally invariant under some iterate of $f_n^j$. After a suitable change of coordinates, we assume without loss of generality that the fibers are $F_0 = \pi^{-1}(0)$ and $F_\infty = \pi^{-1}(\infty)$.  By the proof of Proposition \ref{prop:Hirzebruch} after replacing $f_n^j$ by another iterate $f_n^{j'}$ with $j' \in \iterates$ we may assume that both of the fibers are totally invariant under $f_n^{j'}$. Indeed, if $j \le 6$ we can let $j' = 2j$. If not, we know by Remark \ref{rem:long-then-Gm-P1-bundle} that the iterate that becomes regular comes from a $\bG_m^2$ endomorphism. So, the corresponding martix of $\Phi^j$ is a diagonal with non-negative entries and hence $F_0$ and $F_\infty$ are both totally invariant. 

$X$ is obtained from $X_n$ after finitely many blow-ups. It follows similar to the previous case that none of the points in $X_n \setminus (F_0 \cup F_\infty)$ are blown-up in the process because that would result in another fiber that is totally invariant under an iterate of $\Phi^{j'}$. Hence, we can only have two types of blow-ups: $(i)$ the ones centered at a divisor above $F_0$ and $(ii)$ the ones centered at a divisor above $F_\infty$. Since $C_0,F_0,$ and $F_\infty$ are totally invariant under $f_n^{j'}$, $f_n^{j'}$ must induce polynomial endomorphisms on $U_1 = X \setminus (C_0 \cup F_0) \simeq \C^2$ and $U_2 = X \setminus (C_0 \cup F_\infty) \simeq \C^2$. Hence, by the proof of the previous we conclude that all type $(i)$ and $(ii)$ exceptional divisors of $X$ are totally invariant under $f_n^{j'}$. So, $f_n^{j'}$ extends to an endomorphism of $X$.    
\end{proof}

\section{The case of \texorpdfstring{$\bP^1$}{P1}-bundles over a curve of genus at least 1}
\label{sec:P1-bundle-over-C}
In this section, we prove Theorem \ref{thm:more-precise} in the case of $\bP^1$-bundles over a curve of genus at least 1. We start with the next Proposition which describes the structure of dominant endomorphisms of isotrivial $\bP^1$-bundles.  
\begin{proposition}
\label{prop:trivial-then-split}
Let $C$ be a smooth curve of genus at least 1 over $\C$ and $\Phi$ be a dominant endomorphism of $C \times \bP^1$. Then, $\Phi$ must be given by a product of the form
\[
(x,y) \mapsto (\varphi_1(x),\varphi_2(y)),
\]
where $\varphi_1$ and $\varphi_2$ are endomorphisms of $C$ and $\bP^1$, respectively. 
\end{proposition}
\begin{proof}
An endomorphism of $C \times \bP^1$ is given by morphisms $\varphi_1: C \times \bP^1 \lra C$ and $\varphi_2: C \times \bP^1 \lra \bP^1$. Let $\pi_1$ and $\pi_2$ be the projections of $C\times \bP^1$ onto $C$ and $\bP^1$, respectively. The proposition is then claiming that $\varphi_1:=\pi_1\circ \Phi$ factors through $\pi_1$ and $\varphi_2:=\pi_2 \circ \Phi$ factors through $\pi_2$. Note that $\varphi_1$ must factor through the projection $\pi_1$ by Lemma \ref{lem:rat-pres-fibers}. So, we only need to show that $\varphi_2$ factors through $\pi_2$.

Suppose that $\varphi_2$ generically contracts fibers of the projection $\pi_1$. Then, $\varphi_2$ must factor through $\pi_1$ as before. This contradicts the assumption that $\Phi$ is a dominant self-map. So, suppose that the morphism $\varphi_2$ is generically dominant and $d$-to-$1$ when restricted to the fibers of the projection $\pi_1$. This is the same data as a morphism from $C$ to $\text{Rat}_d(\bP^1)$ where $\text{Rat}_d(\bP^1)$ denotes the space of self-maps of $\bP^1$ of degree $d$. $\text{Rat}_d(\bP^1)$ is affine and there can be no non-constant morphisms from a projective curve to an affine space. We conclude that the map $\varphi_2$ has no dependence on $C$, and thus, must factor through $\pi_2$. 
\end{proof}

We are now ready to prove the main theorem of this section. 
\begin{theorem}
\label{thm:P1-bundle-over-C}
Suppose we are in the situation of Theorem \ref{thm:more-precise} and $X$ is a $\bP^1$-bundle over a curve $C$ of genus at least $1$. Then, $\Phi$ itself must be regular.
\end{theorem}
\begin{proof}
Let $\pi: X \lra C$ be the projection as defined in Section \ref{sec:facts}. Let $g := \text{genus}(C)$. By Lemma \ref{lem:rat-pres-fibers} we know that there exists an endomorphism $a:C \lra C$ such that $a\circ \pi = \pi \circ \Phi$. If $\Phi$ is regular we would be done so we assume not. Then, by Proposition \ref{prop:totally-inv} there must exist subsets $W$ such that $\Phi^{-k}(W) = W$. At the expense of replacing $k$ by a mutiple we assume that $\Phi^{-k}(p) = p$ for all $p \in W$. Since $a^k \circ \pi = \pi \circ \Phi^k$ we must also have $a^{-k}(\pi(p)) = \pi(p)$ for all $p \in W$. Because $g \ge 1$, $a$ is \'etale and has a totally invariant point so it must be an automorphism. It then follows that $a$ must be of finite order since the curve $C$ has genus at least 1 (See \cite[Chapter IV exercise 5.2]{Hartshorne} and \cite[Theorem 10.1]{Silverman}). In other words, $a^r = Id$ for some $r \ge 1$. Lastly, suppose the indeterminacy locus of $\Phi$ is contained in $\pi^{-1}(S)$ for some finite set $S \subset C$ and let $S'$ be the $a$ orbit of $S$. Then $\Phi$ induces an endomorphism of $X \setminus V$ where $V := \pi^{-1}(S')$. 

So, $\Phi^{rk}$ can be viewed as a $C$-endomorphism of $X$. Since we assumed $\Phi^k$ is not an automorphism, \cite[Theorem 1]{AmerikProjBundles} yields that $X$ becomes trivial after a finite base change. In other words, there exists a finite (\'etale) map of curves $h: E \lra C$ such that $X' := X \times_C E \simeq E \times \bP^1$. To be more precise, if we let $G = \text{Aut}_C(E)$ and let $\rho: G \inj \text{Aut}(\bP^1)$ be an injective group homomorphism so that $G$ induces a natural action on $E \times \bP^1$ given by 
\begin{equation}
\label{eqn:action-of-g}
(x,y) \mapsto (\g(x), \rho(\g)(y))
\end{equation}
for all $\g\in G$, then $X$ is the quotient of $E \times \bP^1$ by the action of $G$ (See the discussion after \cite[Theorem 1]{AmerikProjBundles}). We will use $\widetilde{\g}$ to denote the automorphism defined in equation \eqref{eqn:action-of-g}.

The endomorphism $\Phi^{rk}$ lifts to an endomorphism $F$ of $X'$ given by $(\Phi^{rk}, \text{Id})$. If we let $\pi_1: X \times _C E = X' \lra X$ be the projection onto the first coordinate, then the next diagram must commute
\begin{equation}
\label{eqn:F-com}
\begin{tikzcd}
X'  \arrow[r, "F"]\arrow[d, "\pi_1"] & X' \arrow[d, "\pi_1"]\\
X \arrow[r, "\Phi^{rk}"] & X.  
\end{tikzcd}
\end{equation}
Since $X' \cong E \times \bP^1$, we conclude by Proposition \ref{prop:trivial-then-split} that $F$ must be split when viewed as an endomorphism of $E \times \bP^1$. Suppose $F$ is given by 
\[
(x,y) \mapsto (g(x), f(y)),    
\]
for $g \in \End(E)$, a rational function $f \in \C(y)$ and all $x \in E$ and $y \in \bP^1$. Then, $g(y)$ must indeed be the identity given that the map $F$ was defined by the pair $(\Phi^{rk}, \Id)$ and thus induces the identity action on the base $E$. So, $F$ is given by
\begin{equation}
\label{eqn:form-of-F^l}
(x,y) \mapsto (x, f(y)).   
\end{equation}
Moreover, since this map is a lift of the endomorphism $\Phi^{rk}$ on $X$ we conclude that it respects the $G$-action on $E \times \bP^1$. In other words, using this along with equations \eqref{eqn:action-of-g} and \eqref{eqn:form-of-F^l} we see that the next diagram commutes for all $\g\in G$
\begin{equation}
\label{eqn:G-equivariant}
\begin{tikzcd}
E \times \bP^1  \arrow[r, "F"]\arrow[d, "\tilde{\g}"] & E \times \bP^1 \arrow[d, "\tilde{\g}"]\\
E \times \bP^1 \arrow[r, "F"] & E \times \bP^1.  
\end{tikzcd}
\end{equation}
Indeed, the image of $(x,y)$ and $(\g(x), \rho(\g)(y))$ under $F$ must be the same modulo the action of $G$. But the images are $(x,f(y))$ and $(\g(x), (f\circ \rho(\g))(y))$, respectively. It follows that $(\g(x), (f\circ \rho(\g))(y)) = \g' \cdot (x, f(y))$  for some $\g' \in G$. Then, it follows from the freeness of the action of $G$ that $\g' = \g$. Hence, 
\[
f \circ \rho(\g) = \rho(\g) \circ f,
\]
and the commutativity of \eqref{eqn:G-equivariant} follows. 

Using equation \eqref{eqn:F-com}, $F^{-1}(\pi_1^{-1}(p)) = \pi_1^{-1}(p)$ for all $p \in W$ since $\Phi^{-k}(p) = p$. By equation \eqref{eqn:form-of-F^l}, $F$ fixes the fibers of $E \times \bP^1 \lra E$ and thus the individual points in $\pi_1^{-1}(p)$ must all be totally invariant under $F$. After a suitable change of coordinates we may assume $(e,\infty) \in \pi_1^{-1}(p)$ for some $e \in E$. This implies that $f$ is a polynomial in $\C[y]$ and the curve $E \times \infty$ is totally invariant under $F$. If $E \times q$ is another horizontal curve that is totally invariant under an iterate of $F$ we assume after a suitable coordinate change that $q = 0$. Also, note that there is at most two totally invariant horizontal curves under an iterate of $F$ because otherwise $f$ would have to be an automoprhism which contradicts the hypothesis that $\Phi^k$ is non-invertible. At the expense of replacing $k$ with a multiple we assume without loss of generality that any curve that is totally invariant under an iterate of $F$ is already totally invariant under $F$ itself.

Define $L$ to be the union of all totally invariant curves in $E \times \bP^1$ under $F$ that are not fibers of the projection $E \times \bP^1 \lra E$. $L$ is clearly non-empty as $E \times \{\infty\} \subset L$. Also, note that $L$ only consists of horizontal curves. And, by the discussion above, it can have at most two irreducible components. Also, by equation \eqref{eqn:G-equivariant} we must have that $L$ is $G$-invariant. Let $I := \pi_1(L)$. Then, we have 

\begin{lemma}
\label{claim:inverse-image-of-I}
$\pi_1^{-1}(I) = L$ and $I$ is either
\begin{enumerate}
    \item irreducible and intersects the fibers of $\pi: X \lra C$  at exactly one point, or
    \item irreducible and intersects the fibers of $\pi: X \lra C$  at exactly two distinct point, or
    \item a union of two disjoint irreducible curves each intersecting the fibers of $\pi: X \lra C$ in exactly one point 
\end{enumerate}
\end{lemma}
\begin{proof}
The inclusion $L \subseteq\pi_1^{-1}(I)$ is clear since $I = \pi_1(L)$. Suppose for the sake of contradiction that $\pi_1^{-1}(I)$ is not contained in $L$. Then, there must exist a point $(x_0,y_0) \in \pi_1^{-1}(I)$ that is not contained in $L$. Since $\pi_1(x_0,y_0) \in \pi_1(L)$, there must exist $\g \in G$ such that $(\g(x_0), \rho(\g)(y_0)) \in L$. So, by our assumption about $L$, $\rho_\g(y_0)$ is either $0$ or $\infty$. It follows that $\widetilde{\g}(E \times y_0) \subseteq L$. Then, by equation \eqref{eqn:G-equivariant} we conclude that $E \times y_0$ must also be totally invariant under $\Phi^{rk}$. So, $E \times y_0 \subseteq L$ which contradicts $(x_0,y_0) \notin L$.

To prove the second part note that if $I$ intersects each fiber at exactly one point then we are clearly in case (1). Now, suppose that $I$ intersects each fiber at exactly two points and $I$ is not irreducible. Then, we have to be in case (3). So, we assume that $I$ is irreducible and intersects the generic fiber in at least $s$ distinct points for some $s \ge 2$. Then using the fact that $L$ is $G$-invariant we see that there must exist $(x_0,y_1), \dots, (x_0,y_s) \in L$ such that $y_1,\dots,y_s$ are distinct points in $\bP^1$. Since $L$ is a union of horizontal curves we conclude that $E \times y_i \subseteq L$ for all $i = 1,\dots,s$. If $s \ge 3$ we get a contradiction using the fact that $L$ has at most two irreducible components. We conclude that $s = 2$, $L$ is the union of $C_1 \times y_0$ and $C \times y_1$ and it is clear that $I$ intersects each fiber at exactly two distinct points.

\end{proof}

Using the lemma, let $I_1$ and $I_2$ be the irreducible components of $I$ with $I_2$ potentially empty. We will now show that $\Phi$ is regular on $X$. We can cover $C$ by open subsets $C_1,\dots,C_n$ such that each $C_i$ is invariant under $a$ and there is an isomorphism $u_i: \pi^{-1}(C_i) \simeq C_i \times \bP^1$ for all $i = 1,\dots,n$. We can do this since for every point $p \in C$, we can choose an affine $C_p$ that contains all of the points in the orbit of $p$ under $a$. Since $C_p$ is affine, $X$ is trivial over $C_p$. If we remove the orbit of all of the points in $C \setminus C_p$ under $a$ from $C$ we get an open subset of $C_p$ that is invariant under $a$ and still contains the orbit of $p$. By compactness of $C$, we cover $C$ by finitely many such $C_p$'s. Similarly, at the expense of replacing $C_i$'s with smaller $a$-invariant curves, we assume that the projection $u_i(I_j)\cap (C_i \times \bP^1) \lra \bP^1$ is not surjective for all $i=1,\dots,n$ and all $j = 1,2$.   

By Claim \ref{claim:inverse-image-of-I} we have $\pi_1^{-1}(I) = L$. So, using equation \eqref{eqn:F-com} we have
$$\pi_1^{-1}(\Phi^{-rk})(I) = F^{-1}(\pi_1^{-1}(I)) = F^{-1}(L) = L = \pi_1^{-1}(I).$$ 
Applying $\pi_1$ to the first and last terms we get $\Phi^{-rk}(I) = I$. 

To prove that $\Phi$ is regular it suffices to show that the induced rational maps $\psi_i: C_i\times \bP^1 \dra C_i\times \bP^1$ by $u_i \circ \Phi \circ u_i^{-1}$ are all regular. We do this only for $\psi_1$ and the proof for the rest of $\psi_i$'s follows analogously. 

As a consequence of Lemma \ref{claim:inverse-image-of-I} we have the next lemma. 
\begin{lemma}
\label{lem:preimage-of-I}
Let $V$ be as defined in the beginning of the proof. Then, $\Phi^{-1}(I \setminus V) \subseteq I \setminus V$.
\end{lemma}
\begin{proof}
We note that $\Phi(I \setminus V) \subset I \setminus V$. Indeed, the Zariski closure of $\Phi(I \setminus V)$ must also be totally invariant under $\Phi^{rk}$. This means that the closure of $\pi_1^{-1}(\Phi(I \setminus V))$ is totally invariant under $F$. Note that $\pi_1^{-1}(\Phi(I \setminus V))$ introduces irreducible components that are totally invariant under an iterate of $\Phi^{rk}$. So, they must be in $L$ by the definition of $L$. If $\Phi(I \setminus V) \not\subseteq I$, this contradicts $\pi_1^{-1}(I) = L$. So, $\Phi(I \setminus V) \subset I \setminus V$. Thus, $\Phi^{-1}(I \setminus V) \subseteq I \setminus V$ using $\Phi^{-rk}(I) = I$ and the fact that $\Phi$ induces an endomorphism on $X \setminus V$.   
\end{proof}

Suppose we are in case (1) of Claim \ref{claim:inverse-image-of-I}. Let $\widetilde{I} = u_1(I \cap \pi^{-1}(C_1))$ and $\widetilde{V} = u_1(V \cap \pi^{-1}(C_1))$. Then, using Lemma \ref{lem:preimage-of-I} we must have $\psi_1^{-1}(\widetilde{I} \setminus \widetilde{V}) \subseteq \widetilde{I} \setminus \widetilde{V}$. Since we are in case (1) of Claim \ref{claim:inverse-image-of-I}, $\widetilde{I}$ intersects the fibers of $\pi_1: C_1 \times \bP^1 \lra C_1$ at exactly one point. Thus, the projection $\pi_1$ sends $\widetilde{I}$ to $C_1$ birationally. As a result, there exists a birational inverse $C_1 \lra C_1 \times \bP^1$ with image equal to $\widetilde{I}$ given by 
\begin{equation}
\label{eqn:form-of-j}
x \mapsto (x, \ell(x)),
\end{equation}
where $\ell$ is a morphism from $C_1$ to $\bP^1$. Recall that the projection $\widetilde{I} \lra \bP^1$ is not surjective by our assumption on $C_1$. By equation \eqref{eqn:form-of-j}, we conclude that $\ell$ is not surjective. After a suitable change of coordinates we can assume that $\infty$ is not in the image of $\ell$. So, we may view $\ell$ as a morphism to $\A^1$. Now consider the automorphism $\iota$ of $C_1 \times \bP^1$ given by 
\[
(x,y) \mapsto \left(x, \frac{1}{y - \ell(x)}\right).
\]
We note that $\iota$ sends $\widetilde{I}$ to $C_1 \times \infty$. To prove $\psi_1$ is regular it suffices to show that $\iota \circ \psi_1 \circ \iota^{-1}$ is regular. So, we may assume without loss of generality that $\widetilde{I} = C_1 \times \infty$. Using $\psi_1^{-1}(\widetilde{I} \setminus \widetilde{V}) \subseteq \widetilde{I} \setminus \widetilde{V}$ along with Lemma \ref{lem:rat-pres-fibers} we may assume that $\psi_1$ is given by 
\[
(x,y) \mapsto (a(x), b(x,y)),
\]
for some rational function $a$ and some $b \in \C(C_1)[y]$. Note that $\psi_1^{rk}$ is regular on $C_1 \times \bP^1$ which lets us deduce that it is of the form 
\[
(x,y) \mapsto (a^{rk}(x), b'(x,y)),
\]
for some $b' \in \C[C][y]$. We conclude using case (a) of Lemma \ref{lem:skew-then-split} that $b \in \C[C_1][y]$. Therefore, $\psi_1$ is regular on $C_1 \times \bP^1$ and we are done. 

Similarly, if we are in case (3) of Lemma \ref{claim:inverse-image-of-I}, there must exist morphisms $\ell_j: C_1 \lra \bP^1$ for $j = 1,2$ such that the map $C_1 \lra C_1 \times \bP^1$ given by
\[
x \mapsto (x, \ell_j(x)),
\]
has image equal to $I_j$ for each $j = 1,2$.  
Considering the change of coordinates given by the automorphism
\[
(x,y) \mapsto \left(x, \frac{y - \ell_1(x)}{y - \ell_2(x)}\right),
\]
we may then assume without loss of generality that the irreducible components of $\widetilde{I}$ are $C_1 \times 0$ and $C_1 \times \infty$ and $\psi_1$ must be given by 
\[
(x,y) \mapsto (a(x), c(x)y^{-d}). 
\]
for some $c \in \C(C_1)$ and some $d \in \Z$ (not necessarily positive). So, $\psi_1^2$ is given by 
\[
(x,y) \mapsto \left(a^2(x), \frac{c(a(x))}{c(x)^d}y^{d^2}\right).
\]
Using case (a) of Lemma \ref{lem:skew-then-split} we conclude that $\frac{c(a(x))}{c(x)^d} \in \C[C_1]^\ast$. Therefore,
\[
\frac{c(a^k(x))}{c(a^{k-1}(x))^d} \cdots \frac{c(a^{k-1}(x))^{d^{k-1}}}{c(x)^{d^k}} = \frac{c(a^k(x))}{c(x)^{d^k}} = \frac{c(x)}{c(x)^{d^k}} = c(x)^{1 - d^k}\in \C[C]^\ast.
\]
Therefore, $c(x) \in \C[C_1]^\ast$. We conclude that $\psi_1$ is already regular on $C_1 \times \bP^1$ and we are done. 

Case (2) of Lemma \ref{claim:inverse-image-of-I} is similar to the previous case but a little more subtle. Choose loops $\gamma_i$ for $i = 1,\dots,4g$ in $C_1$ such that $C_i \setminus (\cup_i \gamma_i)$ is simply connected (See \cite[Page 5]{hatcher}). Let $\Gamma := \cup_i \gamma_i$. Similar to the previous cases, after a suitable change of coordinates we assume $\infty$ is not in the image of $\widetilde{I}$ under the projection $C_1 \times \bP^1 \lra \bP^1$. Then, using Lemma \ref{claim:inverse-image-of-I} along with the fact that $C_1 \setminus \Gamma$ is simply connected, there must exist analytic functions $\ell_j: C_1 \setminus \Gamma \lra \C$ such that 
\[
(x, \ell_1(x)) \text{ and } (x, \ell_2(x)), 
\]
are the two distinct analytic branches of $\widetilde{I}$ in $(C_1 \setminus \Gamma) \times \bP^1$. Using the biholomorphic change of coordinates 
\[
(x,y) \mapsto \left(x, \frac{y - \ell_1(x)}{y - \ell_2(x)}\right),
\]
and arguing exactly as in the previous case we see that $\psi_1$ is given by 
\[
(x,y) \mapsto (a(x), c(x)y^{-d}),
\]
for some $d \in \Z$ and some $c \in \C[C_1 \setminus \Gamma]^\ast$. If we let $\OO_a(\Gamma)$ denote the $a$-orbit of $\Gamma$, we conclude that $\psi_1$ is a holomorphic endomorphism of $(C_1 \setminus \OO_a(\Gamma)) \times \bP^1$. But, $\Gamma$ is arbitrary and given any point $p \in C_1$ we can choose $\Gamma$ so that the $a$-orbit of $p$ (which consists of finitely many points) does not intersect $\Gamma$. Therefore, by compactness, $C_1$ can be covered by finitely many sets of the form $C_1 \setminus \OO_a(\Gamma)$ such that $\psi_1$ induces an endomorphism of $(C_1 \setminus \OO_a(\Gamma)) \times \bP^1$. Hence, $\psi_1$ is holomorphic as a self-map of $C_1 \times \bP^1$. It follows from the fact that $\psi_1$ is a rational function that $\psi_1$ must be regular everywhere and thus is an endomorphism of $C_1 \times \bP^1$. 

We conclude in an identical manner that all $\psi_i$ induce endomorphisms of $C_i \times \bP^1$ for $i = 1,\dots,n$. Therefore, $\Phi$ must be an endomorphism on $X$ by the discussion right before Lemma \ref{lem:preimage-of-I}.   
\end{proof}
\section{Proof of Theorems \ref{thm:main} and \ref{thm:more-precise}}
\label{sec:proof-main-thms}
\begin{proof}[Proof of Theorem \ref{thm:main}]
The first conclusion of Theorem \ref{thm:main} is a clear consequence of Theorem \ref{thm:P2}. The second conclusion follows from the discussion in Remark \ref{rem:examples-for-second-part}. 
\end{proof}
\begin{proof}[Proof of Theorem \ref{thm:more-precise}]
  Suppose that $\kappa(X) \ge 0$. By Theorem \ref{thm:pos-kod}, $\Phi^k$ must be \'etale. If $\Phi$ itself is regular we are done. If not, we conclude by Proposition \ref{prop:totally-inv} that there exists a subset $W$ of $X$ of codimension 2 such that $\Phi^{-k}(W) = W$. In particular, there exists $\ell \ge 1$ such that $\Phi^{-k\ell}(w) = w$ for all $w \in W$. Since $\Phi^{k\ell}$ is \'etale we conclude that $\Phi^{k\ell}$ must be an automorphism which contradicts the hypothesis that $\Phi^k$ is a non-invertible endomorphism. This finishes the proof of Theorem \ref{thm:more-precise} in the case of $\kappa(X) \ge 0$. 

Now assume that $\kappa(X) = -\infty$. It follows from Theorem \ref{thm:ruled} that $X$ is either a toric surface or a $\bP^1$-bundle over a smooth curve of genus at least 1. We conclude the proof by Theorems \ref{thm:toric} and \ref{thm:P1-bundle-over-C}. 
\end{proof}

\section{Proof of Theorem \ref{thm:birational}}
\label{sec:birational}
 Before proving Theorem \ref{thm:birational} we first need to give a classification similar to Theorem \ref{conj:lin-conj} of invariant subvarieties of $\bP^2$ under automorphisms that do not preserve non-constant fibrations. More specifically, we prove Proposition \ref{prop:inv-subvarieties} which provides this classification for all $\bP^n$ and $n \ge 2$. Before stating this proposition, we need the next definition. 

 \begin{definition}
Let $\Phi: \bP^n \lra \bP^n$ be an endomorphism. We let $IH(\Phi)$ denote the union of all irreducible hypersurfaces $H$ such that $\Phi(H) = H$. 
\end{definition}

\begin{proposition}
\label{prop:inv-subvarieties}
Suppose that $\Phi$ is an automorphism of $\bP^n$ given by a matrix $A \in \text{PGL}(n + 1, \C)$. Moreover, assume there does not exist a non-constant rational fibration $f: \bP^n \dra \bP^1$ such that $f \circ \Phi = f$. Then, one of the following must hold:
\begin{itemize}
    \item [(a)] After a suitable change of coordinates $\Phi$ is given by a matrix of the form $$J_{\lambda_1, 2} \bigoplus J_{\lambda_2, 1} \bigoplus \cdots \bigoplus J_{\lambda_{n-1},1} \bigoplus J_{1,1}$$
    where $\lambda_1, \dots, \lambda_{n-1}$ are multiplicatively independent and $IH(\Phi) = V(X_1X_2\cdots X_n)$ or, 

    \item[(b)] $A$ is digonalizable of the form 
    \[
    J_{\lambda_1, 1} \bigoplus J_{\lambda_2, 1} \bigoplus \cdots \bigoplus J_{\lambda_{n},1} \bigoplus J_{1,1}
    \]
    with multiplicatively independent eigenvalues $\lambda_1, \dots, \lambda_{n}$ and $IH(\Phi) = V(X_0X_1\cdots X_n)$.

\end{itemize}
\end{proposition}
\begin{proof}
We may assume after a suitable change of coordinates that $A$ is in Jordan canonical form. Our first step is to show that $A$ is either diagonalizable or of the form $$J_{\lambda_1, 2} \bigoplus J_{\lambda_2,1} \bigoplus \cdots \bigoplus J_{\lambda_{n},1}.$$ This is a consequence of the proofs of Cases 2 and 3 of \cite[Theorem 2.1]{GH19}. More specifically, by the proof of Cases 2 and 3 of \cite[Theorem 2.1]{GH19}, having a Jordan block of length at least 3 or two Jordan blocks of length two would yield a non-constant rational fibration that is preserved by $\Phi$. This contradicts the hypothesis of the proposition. 

We now assume after normalization that the eigenvalue corresponding to the last Jordan block is equal to 1 since $A$ is a projective matrix. Now, suppose that $A$ is of the form $J_{\lambda_2, 2} \bigoplus J_{\lambda_2,1}\bigoplus \cdots \bigoplus J_{\lambda_{n-1},1} \bigoplus J_{1,1}$. This means that $\Phi$ viewed as a map on $\A^n \cong \bP^n \setminus V(X_n)$ is given by 
\[
(x_1,\dots,x_n) \mapsto (\lambda_1x_1 + x_2, \lambda_1 x_2, \lambda_2x_3,\dots,\lambda_{n-1}x_n).
\]
We may assume that the eigenvalues $\lambda_1,\dots,\lambda_{n-1}$ are multiplicatively independent because otherwise $\Phi$ will preserve a non-constant rational fibration (See \cite[Theorem 2.1]{GH19}, specifically the discussion before Case 1 and the proof of Case 1). 

Now, since $A = J_{\lambda_2, 2} \bigoplus J_{\lambda_2,1}\bigoplus \cdots \bigoplus J_{\lambda_{n-1},1} \bigoplus J_{1,1}$ we clearly must have $$V(X_1X_2\cdots X_n) \subseteq IH(\Phi).$$ To prove the reverse inclusion suppose that $H$ is a hypersurface in $IH(\Phi) \setminus V(X_1X_2\cdots X_n)$. $H$ can be viewed as a hypersurface of $\A^n \cong \bP^n \setminus V(X_n)$ defined by a polynomial $F(x_1,\dots,x_n) \in \C[x_1,\dots,x_n]$. We can write 
\[
F(x_1,\dots,x_n) = p_d(x_2,\dots,x_n)x_1^d + \cdots + p_0(x_2,\dots,x_n),
\]
for $p_0,\dots,p_d \in \C[x_2,\dots,x_n]$ and some $d\ge 0$. Since $H \in IH(\Phi) \setminus V(X_1X_2\cdots X_n)$, we can pick a point $\alpha := (a_1, \dots, a_n)$ in $H$ such that $a_2 \cdots a_n \ne 0$. Using $\Phi(H) = H$, the orbit of $\alpha$ which is the set 
\[
\{(\lambda_1^ma_1 + m\lambda_1^{m-1}a_2, \lambda_1^m a_2, \lambda_2^ma_3, \dots, \lambda_{n-1}^ma_n): m \ge 0\} 
\]
must be contained in $H$. Therefore, 
\begin{align}
\label{eqn:F-zero}
0 &= F(\lambda_1^ma_1 + m\lambda_1^{m-1}a_2, \lambda_1^m a_2, \lambda_2^ma_3, \dots, \lambda_{n-1}^ma_n) \\ 
&= \sum_{i = 0}^d p_i(\lambda_1^m a_2, \lambda_2^ma_3, \dots, \lambda_{n-1}^ma_n)(\lambda_1^ma_1 + m\lambda_1^{m-1}a_2)^i \notag \\
&= m^d\lambda_1^{d(m-1)}a_2^d p_d(\lambda_1^ma_2, \dots , \lambda_{n-1}^ma_n) + \sum_{i=0}^{d-1}m^iQ_i(\lambda_1^ma_1, \lambda_1^ma_2, \dots , \lambda_{n-1}^ma_n), \notag 
\end{align}
for some polynomials $Q_i$ with complex coefficients. Now note that we can write
\begin{equation}
\label{eqn:first-term}
m^d\lambda_1^{d(m-1)}a_2^d p_d(\lambda_1^ma_2, \dots , \lambda_{n-1}^ma_n) = \sum_{1 \le j \le s} \alpha_j m^d r_j^m,
\end{equation}
for some $r_1,\dots,r_s \in \C$ that are not roots of unity and such that $r_i/r_j$ is also not a root of unity for all $i \ne j$ using the fact that $\lambda_1,\dots,\lambda_{n-1}$ are multiplicatively independent. Moreover, we note that $\alpha_1\cdots \alpha_s \ne 0$ using the fact that $a_2\cdots a_n \ne 0$.  Similarly, we can write
\begin{equation}
\label{eqn:second-term}
\sum_{i=0}^{d-1}m^iQ_i(\lambda_1^ma_1, \lambda_1^ma_2, \dots , \lambda_{n-1}^ma_n) = \sum_{1\le k \le t} R_k(m)(r'_k)^m,
\end{equation}
for distinct $r'_1,\dots,r'_t \in \C$ and polynomials $R_1,\dots,R_t \in \C[m]$ of degree at most $d-1$. Moreover, at the expense of assuming $m$ is always a multiple of a suitably large fixed integer, say $\ell$, and replacing $r_1,\dots,r_s,r'_1,\dots,r'_t$ with $r_1^\ell,\dots,r_s^\ell,(r'_1)^\ell,\dots,(r'_t)^\ell$ we may assume that if $r'_i$ is a root of unity then $r'_i = 1$. Putting equations \eqref{eqn:F-zero}, \eqref{eqn:first-term}, and \eqref{eqn:second-term} together we get 
\begin{align}
F(\lambda_1^ma_1 + m\lambda_1^{m-1}a_2, \lambda_1^m a_2, \lambda_2^ma_3, \dots, \lambda_{n-1}^ma_n) &= \sum_{1 \le j \le s} \alpha_j m^d r_j^m + \sum_{1 \le k \le t} R_k(m)(r'_k)^m \notag \\ 
&= \sum_{1 \le j \le s} q_j(m) r_j^m + \sum_{k \in S^c} R_k(m)(r'_k)^m 
\end{align}
where $S$ is the subset of indices of $r'_i$'s that are equal to some $r_j$ and $q_j(m)$ are polynomials obtained after combining all the terms in $\sum_{1 \le j \le s} q_j(m) r_j^m + \sum_{k \in S^c} R_k(m)(r'_k)^m$ that correspond to $r_j^m$. Note that all $q_1,\dots,q_s$ are polynomials of degree $d$. Thus, the  equation must be the general equation for a non-degenerate linear recurrence sequence (See Section 3 of \cite{GH16}). Then, by \cite[Proposition 3.2]{GH16}, Equation \eqref{eqn:F-zero} can only be equal to zero for finitely many values of $m$. This is a contradiction. Therefore, $V(X_1\cdots X_n) = IH(\Phi)$ and we are done. The proof is almost identical when $A$ is diagonalizable. 
\end{proof}

We are now ready to prove Theorem \ref{thm:birational}.
\begin{proof}[Proof of Theorem \ref{thm:birational}]
If $\Phi$ is regular we are done. So, we assume this is not the case and $\Phi^k$ is regular for some $k \ge 2$. Since $\Phi$ is birational, $\Phi^k$ must be an automorphism of $\bP^2$. So, it corresponds to a matrix $A$ in $\text{PGL}(3, \C)$. We assume that $\Phi$ does not preserve a non-constant rational fibration and show that $\Phi$ itself must be an automorphism of $\bP^2$. After a suitable change of coordinates we may assume that $A$ is in Jordan normal form. Using proposition \ref{prop:inv-subvarieties} we assume $A$ is either of the form $J_{\lambda_1,1} \bigoplus J_{\lambda_{2},1} \bigoplus J_{1,1}$ or $J_{\lambda_1, 2} \bigoplus J_{1,1}$. We handle these two cases separately.

\textbf{Case 1. $A = J_{\lambda_1,1} \bigoplus J_{\lambda_{2},1} \bigoplus J_{1,1}$.} Recall that $IH(\Phi) = V(X_0X_1X_2)$ by Proposition \ref{prop:inv-subvarieties}. Using Lemma \ref{lem:add-more-totally-invs} we conclude that $\Phi$ induces an endomorphism of $\bP^2 \setminus V(X_0X_1X_2) \cong 
\bG_m^2$. So, there is a matrix $B \in M_{2,2}(\Z)$ and some $\vec{v} \in \C^2$ such that $\Phi$ as an endomorphism of $\bG_m^2$ is given by 
\begin{equation}
\label{eqn:form-of-phi-1}
\vec{x} \mapsto \vec{v} + B\vec{x}.
\end{equation}
Therefore, $\Phi^k$ must be given by 
\begin{equation}
\label{eqn:phi^k-1}
\vec x \mapsto \vec v + B \vec v + \cdots + B^{k-1}\vec v + B^{k}\vec x  
\end{equation}
On the other hand, if we let $\vec \lambda = (\lambda_1,\lambda_2)$, because of the form of the matrix $A$, $\Phi^k$ on $\bG_m^2$ is given by
\begin{equation}
\label{eqn:phi^k-2}
\vec x \mapsto \lambda + \vec x,
\end{equation}
Putting equations \eqref{eqn:phi^k-1} and \eqref{eqn:phi^k-2} together we conclude that $B^k = \text{Id}$ and
\[
\vec v + B \vec v + \cdots + B^{k-1}\vec v = \vec \lambda.
\]
Multiplying both sides by $(B - \text{Id})$ we conclude that 
\[
(B - \text{Id})\vec \lambda = \vec{0}.
\]
If $B$ is not the identity matrix, this means that the eigenvalues $\lambda_1$ and $\lambda_{2}$ are multiplicatively dependent which is a contradiction. Hence, $B = \text{Id}$ and we conclude that $\Phi$ itself must extend to an endomorphism of $\bP^2$.

\textbf{Case 2. $A = J_{\lambda_1, 2} \bigoplus J_{1,1}$.} In this case $H = V(X_1X_2)$ and $\lambda_1$ is not a root of unity by Proposition \ref{prop:inv-subvarieties}. Analogous to Case 1, we can then conclude that $\Phi$ induces an isomorphism of $\bP^2 \setminus V(X_1X_2) \cong \A^2\setminus \{y = 0\}$ where we let $x$ and $y$ be the coordinates of $\A^2$. It follows that $\Phi$ viewed as a rational self-map of $\A^2$ is given by 
\[
(x,y) \mapsto (p(x,y), cy^d),
\]
for some $c \in \C^\ast$, $p(x,y) \in \C[x,y,y^{-1}]$, and $d \in \{1,-1\}$. If $d = -1$, it follows that $\Phi$ preserves the fibration
\[
(x,y) \mapsto \frac{y}{c} + \frac{1}{y},
\]
which contradicts the hypothesis. So, $d = 1$. Using Case (c) of Lemma \ref{lem:skew-then-split} we conclude that $p(x,y)$ belongs to $\C[x,y]$. Since $\Phi$ is birational on $\C^2$ we conclude that $p(x,y) = ax + q(y)$ for some $a \in \C^\ast$ and $q \in \C[y]$. So, $\Phi$ is given by 
\begin{equation}
\label{eqn:form-of-phi-2}
(x,y) \mapsto (ax + q(y), cy).
\end{equation}
To finish the proof we must show that $\deg(q) \le 1$. Using the fact that $\Phi^k$ is an automorphism of $\bP^2$ we know that $\Phi^k$ induces a linear automorphism of $\A^2$. On the other hand, we can compute that $\Phi^k$ is given by 
\[
(x,y) \mapsto (a^kx + a^{k-1}q(y) + a^{k-2}q(cy) + \cdots + q(c^{k-1}y), c^ky).
\]
Now suppose for the sake of contradiction that $q$ is not linear and let its leading term be $my^j$ for some $j \ge 2$ and $m \in \C^\ast$. Then, since $\Phi^k$ is a linear automorphism, the coefficient of $y^j$ in $a^{k-1}q(y) + a^{k-2}q(cy) + \cdots + q(c^{k-1}y)$ must be zero. This coefficient can be computed as 
\[
m(a^{k-1} + a^{k-2}c^j + \cdots + c^{j(k-1)}) = m\frac{a^k - c^{jk}}{a- c^j} = 0.
\]
This implies that
\begin{equation}
\label{eqn:a,c-1}
a^k = c^{kj}
\end{equation}
Recall that we know from the form of the matrix $A$ that 
\begin{equation}
\label{eqn:a,c-2}
a^k = c^k = \lambda_1    
\end{equation}
Dividing Equation \eqref{eqn:a,c-1} by Equation \eqref{eqn:a,c-2} we get $c^{k(j-1)} = 1$. So, $c$ is a root of unity since $j \ge 2$. Then, $\lambda_1$ is also a root of unity as $\lambda_1 = c^k$. This contradicts our hypothesis about $A$. Therefore, $j = 1$ and $\Phi$ is a linear automorphism and must extend to $\bP^2$. 
\end{proof}

\bibliography{main}

@incollection {AR86,
    AUTHOR = {Artin, M.},
     TITLE = {N\'eron models},
 BOOKTITLE = {Arithmetic geometry ({S}torrs, {C}onn., 1984)},
     PAGES = {213--230},
 PUBLISHER = {Springer, New York},
      YEAR = {1986},
      ISBN = {0-387-96311-1},
   MRCLASS = {11G10},
  MRNUMBER = {861977},
}

@article {diller-favre,
    AUTHOR = {Diller, J. and Favre, C.},
     TITLE = {Dynamics of bimeromorphic maps of surfaces},
   JOURNAL = {Amer. J. Math.},
  FJOURNAL = {American Journal of Mathematics},
    VOLUME = {123},
      YEAR = {2001},
    NUMBER = {6},
     PAGES = {1135--1169},
      ISSN = {0002-9327,1080-6377},
   MRCLASS = {32H50 (37F10)},
  MRNUMBER = {1867314},
MRREVIEWER = {Serge\ Cantat},
       URL =
              {http://muse.jhu.edu/journals/american_journal_of_mathematics/v123/123.6diller.pdf},
}

@book {Bea91,
    AUTHOR = {Beardon, Alan F.},
     TITLE = {Iteration of rational functions},
    SERIES = {Graduate Texts in Mathematics},
    VOLUME = {132},
      NOTE = {Complex analytic dynamical systems},
 PUBLISHER = {Springer-Verlag, New York},
      YEAR = {1991},
     PAGES = {xvi+280},
      ISBN = {0-387-97589-6},
   MRCLASS = {30D05 (30-01 58F23)},
  MRNUMBER = {1128089},
MRREVIEWER = {A.\ \`E.\ Eremenko},
       DOI = {10.1007/978-1-4612-4422-6},
       URL = {https://doi.org/10.1007/978-1-4612-4422-6},
}

@article {BGR,
    AUTHOR = {Bell, Jason and Ghioca, Dragos and Reichstein, Zinovy},
     TITLE = {Rational self-maps with a regular iterate on a semiabelian
              variety},
   JOURNAL = {J. Number Theory},
  FJOURNAL = {Journal of Number Theory},
    VOLUME = {260},
      YEAR = {2024},
     PAGES = {103--119},
      ISSN = {0022-314X,1096-1658},
   MRCLASS = {14K12 (37P55)},
  MRNUMBER = {4712777},
       DOI = {10.1016/j.jnt.2024.01.007},
       URL = {https://doi.org/10.1016/j.jnt.2024.01.007},
}

@article{FS,
  title={Complex dynamics in higher dimension. II},
  author={Fornaess, John Erik and Sibony, Nessim},
  journal={Modern methods in complex analysis (Princeton, NJ, 1992)},
  volume={137},
  pages={135--182},
  year={1995},
  publisher={Princeton University Press Princeton, NJ}
}

@article {FJ07,
    AUTHOR = {Favre, Charles and Jonsson, Mattias},
     TITLE = {Dynamical compactifications of {${\bf C}^2$}},
   JOURNAL = {Ann. of Math. (2)},
  FJOURNAL = {Annals of Mathematics. Second Series},
    VOLUME = {173},
      YEAR = {2011},
    NUMBER = {1},
     PAGES = {211--248},
      ISSN = {0003-486X,1939-8980},
   MRCLASS = {32H50 (32J05 32U35 37F10)},
  MRNUMBER = {2753603},
MRREVIEWER = {Araceli\ Bonifant},
       DOI = {10.4007/annals.2011.173.1.6},
       URL = {https://doi.org/10.4007/annals.2011.173.1.6},
}

@incollection {Jon11,
    AUTHOR = {Jonsson, Mattias},
     TITLE = {Dynamics of {B}erkovich spaces in low dimensions},
 BOOKTITLE = {Berkovich spaces and applications},
    SERIES = {Lecture Notes in Math.},
    VOLUME = {2119},
     PAGES = {205--366},
 PUBLISHER = {Springer, Cham},
      YEAR = {2015},
      ISBN = {978-3-319-11028-8; 978-3-319-11029-5},
   MRCLASS = {37P50 (14G22 26E30)},
  MRNUMBER = {3330767},
MRREVIEWER = {Benjamin\ A.\ Hutz},
       DOI = {10.1007/978-3-319-11029-5\_6},
       URL = {https://doi.org/10.1007/978-3-319-11029-5_6},
}

@article {Nak02,
    AUTHOR = {Nakayama, Noboru},
     TITLE = {Ruled surfaces with non-trivial surjective endomorphisms},
   JOURNAL = {Kyushu J. Math.},
  FJOURNAL = {Kyushu Journal of Mathematics},
    VOLUME = {56},
      YEAR = {2002},
    NUMBER = {2},
     PAGES = {433--446},
      ISSN = {1340-6116,1883-2032},
   MRCLASS = {14J26},
  MRNUMBER = {1934136},
MRREVIEWER = {Sandra\ Di Rocco},
       DOI = {10.2206/kyushujm.56.433},
       URL = {https://doi.org/10.2206/kyushujm.56.433},
}

@book {Hatcher,
    AUTHOR = {Hatcher, Allen},
     TITLE = {Algebraic topology},
 PUBLISHER = {Cambridge University Press, Cambridge},
      YEAR = {2002},
     PAGES = {xii+544},
      ISBN = {0-521-79160-X; 0-521-79540-0},
   MRCLASS = {55-01 (55-00)},
  MRNUMBER = {1867354},
MRREVIEWER = {Donald\ W.\ Kahn},
}

@article {favre,
    AUTHOR = {Favre, Charles and Kuznetsova, Alexandra},
     TITLE = {Families of automorphisms on abelian varieties},
   JOURNAL = {Math. Ann.},
  FJOURNAL = {Mathematische Annalen},
    VOLUME = {391},
      YEAR = {2025},
    NUMBER = {1},
     PAGES = {1147--1197},
      ISSN = {0025-5831,1432-1807},
   MRCLASS = {14K20 (14E05 37F80)},
  MRNUMBER = {4846808},
       DOI = {10.1007/s00208-024-02943-4},
       URL = {https://doi.org/10.1007/s00208-024-02943-4},
}

@article {Iitaka,
    AUTHOR = {Iitaka, Shigeru},
     TITLE = {Logarithmic forms of algebraic varieties},
   JOURNAL = {J. Fac. Sci. Univ. Tokyo Sect. IA Math.},
  FJOURNAL = {Journal of the Faculty of Science. University of Tokyo.
              Section IA. Mathematics},
    VOLUME = {23},
      YEAR = {1976},
    NUMBER = {3},
     PAGES = {525--544},
      ISSN = {0040-8980},
   MRCLASS = {14C30 (14E05 32J25)},
  MRNUMBER = {429884},
MRREVIEWER = {Miles\ Reid},
}

@book {Silverman,
    AUTHOR = {Silverman, Joseph H.},
     TITLE = {The arithmetic of elliptic curves},
    SERIES = {Graduate Texts in Mathematics},
    VOLUME = {106},
 PUBLISHER = {Springer-Verlag, New York},
      YEAR = {1986},
     PAGES = {xii+400},
      ISBN = {0-387-96203-4},
   MRCLASS = {11G05 (14Gxx 14K07 14K15)},
  MRNUMBER = {817210},
MRREVIEWER = {Robert\ S.\ Rumely},
       DOI = {10.1007/978-1-4757-1920-8},
       URL = {https://doi.org/10.1007/978-1-4757-1920-8},
}

@article {ML24,
    AUTHOR = {Matsuzawa, Yohsuke and Wang, Long},
     TITLE = {Arithmetic degrees and {Z}ariski dense orbits of
              cohomologically hyperbolic maps},
   JOURNAL = {Trans. Amer. Math. Soc.},
  FJOURNAL = {Transactions of the American Mathematical Society},
    VOLUME = {377},
      YEAR = {2024},
    NUMBER = {9},
     PAGES = {6311--6340},
      ISSN = {0002-9947,1088-6850},
   MRCLASS = {37P15 (37P05 37P30 37P55)},
  MRNUMBER = {4855313},
MRREVIEWER = {Ruofan\ Li},
       DOI = {10.1090/tran/9211},
       URL = {https://doi.org/10.1090/tran/9211},
}

@incollection {cantat-aut,
    AUTHOR = {Cantat, Serge},
     TITLE = {Dynamics of automorphisms of compact complex surfaces},
 BOOKTITLE = {Frontiers in complex dynamics},
    SERIES = {Princeton Math. Ser.},
    VOLUME = {51},
     PAGES = {463--514},
 PUBLISHER = {Princeton Univ. Press, Princeton, NJ},
      YEAR = {2014},
      ISBN = {978-0-691-15929-4},
   MRCLASS = {37F45 (37F50)},
  MRNUMBER = {3289919},
MRREVIEWER = {Eric\ Bedford},
}

@book {Hartshorne,
    AUTHOR = {Hartshorne, Robin},
     TITLE = {Algebraic geometry},
    SERIES = {Graduate Texts in Mathematics},
    VOLUME = {No. 52},
 PUBLISHER = {Springer-Verlag, New York-Heidelberg},
      YEAR = {1977},
     PAGES = {xvi+496},
      ISBN = {0-387-90244-9},
   MRCLASS = {14-01},
  MRNUMBER = {463157},
MRREVIEWER = {Robert\ Speiser},
}

@article {EGAII,
    AUTHOR = {Grothendieck, A.},
     TITLE = {\'El\'ements de g\'eom\'etrie alg\'ebrique. {II}. \'Etude
              globale \'el\'ementaire de quelques classes de morphismes.},
   JOURNAL = {Inst. Hautes \'Etudes Sci. Publ. Math.},
  FJOURNAL = {Institut des Hautes \'Etudes Scientifiques. Publications
              Math\'ematiques},
    NUMBER = {8},
      YEAR = {1961},
     PAGES = {222},
      ISSN = {0073-8301,1618-1913},
   MRCLASS = {14.55},
  MRNUMBER = {217084},
       URL = {http://www.numdam.org/item?id=PMIHES_1961__8__222_0},
}

@article {Fujimoto,
    AUTHOR = {Fujimoto, Yoshio},
     TITLE = {Endomorphisms of smooth projective 3-folds with non-negative
              {K}odaira dimension},
   JOURNAL = {Publ. Res. Inst. Math. Sci.},
  FJOURNAL = {Kyoto University. Research Institute for Mathematical
              Sciences. Publications},
    VOLUME = {38},
      YEAR = {2002},
    NUMBER = {1},
     PAGES = {33--92},
      ISSN = {0034-5318,1663-4926},
   MRCLASS = {14J30 (14E30)},
  MRNUMBER = {1873169},
MRREVIEWER = {Marco\ Andreatta},
       URL = {http://projecteuclid.org/euclid.prims/1145476416},
}

@article {hor17,
    AUTHOR = {H{\"o}ring, Andreas},
     TITLE = {Totally invariant divisors of endomorphisms of projective
              spaces},
   JOURNAL = {Manuscripta Math.},
  FJOURNAL = {Manuscripta Mathematica},
    VOLUME = {153},
      YEAR = {2017},
    NUMBER = {1-2},
     PAGES = {173--182},
      ISSN = {0025-2611,1432-1785},
   MRCLASS = {14C20 (14F10)},
  MRNUMBER = {3635979},
MRREVIEWER = {Elisa\ Postinghel},
       DOI = {10.1007/s00229-016-0881-8},
       URL = {https://doi.org/10.1007/s00229-016-0881-8},
}

@article {eigenvaluations,
    AUTHOR = {Favre, Charles and Jonsson, Mattias},
     TITLE = {Eigenvaluations},
   JOURNAL = {Ann. Sci. \'Ecole Norm. Sup. (4)},
  FJOURNAL = {Annales Scientifiques de l'\'Ecole Normale Sup\'erieure.
              Quatri\`eme S\'erie},
    VOLUME = {40},
      YEAR = {2007},
    NUMBER = {2},
     PAGES = {309--349},
      ISSN = {0012-9593},
   MRCLASS = {37F10 (32H50)},
  MRNUMBER = {2339287},
MRREVIEWER = {Romain\ Dujardin},
       DOI = {10.1016/j.ansens.2007.01.002},
       URL = {https://doi.org/10.1016/j.ansens.2007.01.002},
}

@inbook{Bea96, place={Cambridge}, series={London Mathematical Society Student Texts}, title={Rational surfaces}, booktitle={Complex Algebraic Surfaces}, publisher={Cambridge University Press}, author={Beauville, Arnaud}, year={1996}, pages={40–54}, collection={London Mathematical Society Student Texts}}

@article {AmerikProjBundles,
    AUTHOR = {Amerik, Ekaterina},
     TITLE = {On endomorphisms of projective bundles},
   JOURNAL = {Manuscripta Math.},
  FJOURNAL = {Manuscripta Mathematica},
    VOLUME = {111},
      YEAR = {2003},
    NUMBER = {1},
     PAGES = {17--28},
      ISSN = {0025-2611,1432-1785},
   MRCLASS = {14J60},
  MRNUMBER = {1981593},
MRREVIEWER = {Cristian\ V.\ Anghel},
       DOI = {10.1007/s00229-002-0347-z},
       URL = {https://doi.org/10.1007/s00229-002-0347-z},
}

@article {Sil96,
    AUTHOR = {Silverman, Joseph H.},
     TITLE = {Rational functions with a polynomial iterate},
   JOURNAL = {J. Algebra},
  FJOURNAL = {Journal of Algebra},
    VOLUME = {180},
      YEAR = {1996},
    NUMBER = {1},
     PAGES = {102--110},
      ISSN = {0021-8693,1090-266X},
   MRCLASS = {11G99},
  MRNUMBER = {1375568},
MRREVIEWER = {Jeffrey\ Lin\ Thunder},
       DOI = {10.1006/jabr.1996.0054},
       URL = {https://doi.org/10.1006/jabr.1996.0054},
}

@article{JW15,
  title={Singular implicit and inverse function theorems. Strong resolution with normally flat centers},
  author={Wlodarczyk, Jaroslaw},
  journal={arXiv preprint arXiv:1510.03480},
  year={2015}
}

@article {GH16,
    AUTHOR = {Ghioca, Dragos},
     TITLE = {The dynamical {M}ordell-{L}ang conjecture in positive
              characteristic},
   JOURNAL = {Trans. Amer. Math. Soc.},
  FJOURNAL = {Transactions of the American Mathematical Society},
    VOLUME = {371},
      YEAR = {2019},
    NUMBER = {2},
     PAGES = {1151--1167},
      ISSN = {0002-9947,1088-6850},
   MRCLASS = {37P55 (11G10 14G17)},
  MRNUMBER = {3885174},
MRREVIEWER = {Joseph\ H.\ Silverman},
       DOI = {10.1090/tran/7261},
       URL = {https://doi.org/10.1090/tran/7261},
}

@article {MSS,
    AUTHOR = {Matsuzawa, Yohsuke and Sano, Kaoru and Shibata, Takahiro},
     TITLE = {Arithmetic degrees and dynamical degrees of endomorphisms on
              surfaces},
   JOURNAL = {Algebra Number Theory},
  FJOURNAL = {Algebra \& Number Theory},
    VOLUME = {12},
      YEAR = {2018},
    NUMBER = {7},
     PAGES = {1635--1657},
      ISSN = {1937-0652,1944-7833},
   MRCLASS = {14G05 (11G35 11G50 14E05 37P05 37P15 37P30)},
  MRNUMBER = {3871505},
MRREVIEWER = {Benjamin\ A.\ Hutz},
       DOI = {10.2140/ant.2018.12.1635},
       URL = {https://doi.org/10.2140/ant.2018.12.1635},
}

@article{GH19,
  title={Density of orbits of endomorphisms of commutative linear algebraic groups},
  author={Ghioca, Dragos and Hu, Fei},
  journal={arXiv preprint arXiv:1803.03928},
  year={2018}
}

@article{cantat-xie,
  title={Birational conjugacies between endomorphisms on the projective plane},
  author={Cantat, Serge and Xie, Junyi},
  journal={arXiv preprint arXiv:2006.00051},
  year={2020}
}

@article {Xie-DML-aut,
    AUTHOR = {Xie, Junyi},
     TITLE = {Dynamical {M}ordell-{L}ang conjecture for birational
              polynomial morphisms on {$\mathbb A^2$}},
   JOURNAL = {Math. Ann.},
  FJOURNAL = {Mathematische Annalen},
    VOLUME = {360},
      YEAR = {2014},
    NUMBER = {1-2},
     PAGES = {457--480},
      ISSN = {0025-5831,1432-1807},
   MRCLASS = {37P55 (14E05 14R10 37P05)},
  MRNUMBER = {3263169},
MRREVIEWER = {Tatiana\ M.\ Bandman},
       DOI = {10.1007/s00208-014-1039-1},
       URL = {https://doi.org/10.1007/s00208-014-1039-1},
}

@article {Xie-DML-End,
    AUTHOR = {Xie, Junyi},
     TITLE = {The dynamical {M}ordell-{L}ang conjecture for polynomial
              endomorphisms of the affine plane},
   JOURNAL = {Ast\'erisque},
  FJOURNAL = {Ast\'erisque},
    NUMBER = {394},
      YEAR = {2017},
     PAGES = {vi+110},
      ISSN = {0303-1179,2492-5926},
      ISBN = {978-2-85629-869-5},
   MRCLASS = {37P05 (14R10)},
  MRNUMBER = {3758955},
MRREVIEWER = {Dragos\ Ghioca},
}

@article {Xie-ZDOC-C2,
    AUTHOR = {Xie, Junyi},
     TITLE = {The existence of {Z}ariski dense orbits for polynomial
              endomorphisms of the affine plane},
   JOURNAL = {Compos. Math.},
  FJOURNAL = {Compositio Mathematica},
    VOLUME = {153},
      YEAR = {2017},
    NUMBER = {8},
     PAGES = {1658--1672},
      ISSN = {0010-437X,1570-5846},
   MRCLASS = {37P55 (32H50)},
  MRNUMBER = {3705271},
MRREVIEWER = {Dragos\ Ghioca},
       DOI = {10.1112/S0010437X17007187},
       URL = {https://doi.org/10.1112/S0010437X17007187},
}

@article {Jonsson-Wulcan-height,
    AUTHOR = {Jonsson, Mattias and Wulcan, Elizabeth},
     TITLE = {Canonical heights for plane polynomial maps of small
              topological degree},
   JOURNAL = {Math. Res. Lett.},
  FJOURNAL = {Mathematical Research Letters},
    VOLUME = {19},
      YEAR = {2012},
    NUMBER = {6},
     PAGES = {1207--1217},
      ISSN = {1073-2780,1945-001X},
   MRCLASS = {37P15 (11G50 37P30)},
  MRNUMBER = {3091603},
       DOI = {10.4310/MRL.2012.v19.n6.a3},
       URL = {https://doi.org/10.4310/MRL.2012.v19.n6.a3},
}

@article {Dang-Tiozzo-limit,
    AUTHOR = {Dang, Nguyen-Bac and Tiozzo, Giulio},
     TITLE = {A central limit theorem for the degree of a random product of
              {C}remona transformations},
   JOURNAL = {Indiana Univ. Math. J.},
  FJOURNAL = {Indiana University Mathematics Journal},
    VOLUME = {72},
      YEAR = {2023},
    NUMBER = {1},
     PAGES = {301--329},
      ISSN = {0022-2518,1943-5258},
   MRCLASS = {37F10 (60F05)},
  MRNUMBER = {4557076},
}

@article{Nak-normal-end,
  title={On the structure of normal projective surfaces admitting non-isomorphic surjective endomorphisms},
  author={Nakayama, Noboru},
  year={2020},
  publisher={Research Institute for Mathematical Sciences, Kyoto University}
}

@article{abboud,
  title={On the dynamics of endomorphisms of affine surfaces},
  author={Abboud, Marc},
  journal={arXiv preprint arXiv:2311.18381},
  year={2023}
}

@article {surf-survey,
    AUTHOR = {Barbaro, Giuseppe and Fagioli, Filippo and R\'ios Ortiz,
              \'Angel David},
     TITLE = {A survey on rational curves on complex surfaces},
   JOURNAL = {Riv. Math. Univ. Parma (N.S.)},
  FJOURNAL = {Rivista di Matematica della Universit\`a{} di Parma. New
              Series. A Journal of Pure and Applied Mathematics},
    VOLUME = {13},
      YEAR = {2022},
    NUMBER = {2},
     PAGES = {505--534},
      ISSN = {0035-6298,2284-2578},
   MRCLASS = {32J15 (14E08 53C55)},
  MRNUMBER = {4579184},
}

@article{Nak-Moishezon,
  title={On normal {M}oishezon surfaces admitting non-isomorphic surjective endomorphisms},
  author={Nakayama, Noboru},
  year={2020},
  publisher={Research Institute for Mathematical Sciences, Kyoto University}
}
\bibliographystyle{alpha}
\end{document}